\documentclass[11pt,reqno]{amsart}
\usepackage{amssymb,latexsym}
\usepackage{graphicx}
\usepackage{mathtools}
\usepackage{color}
\usepackage[T1]{fontenc}
\usepackage{hyperref}
\usepackage{amsmath}
\usepackage{mathdots}
\usepackage[shortlabels]{enumitem}
\usepackage{breqn}
\usepackage{url}    %does nice formatting of URLs
\usepackage{breqn}
\newcommand{\bburl}[1]{\textcolor{blue}{\url{#1}}}

\usepackage{tikz}
\usepackage{tkz-tab}
\usetikzlibrary{decorations.pathreplacing}
\usetikzlibrary{shapes.geometric,positioning}

\calclayout

\makeatletter
\newcommand{\monthyear}[1]{%
  \def\@monthyear{\uppercase{#1}}}
\newcommand{\volnumber}[1]{%
  \def\@volnumber{\uppercase{#1}}}
\AtBeginDocument{%
\def\ps@plain{\ps@empty
  \def\@oddfoot{\@monthyear \hfil \thepage}%
  \def\@evenfoot{\thepage \hfil \@volnumber}}
\def\ps@firstpage{\ps@plain}
\def\ps@headings{\ps@empty
  \def\@evenhead{%
    \setTrue{runhead}%
    \def\thanks{\protect\thanks@warning}%
    \uppercase{\ }\hfil}%
  \def\@oddhead{%
    \setTrue{runhead}%
    \def\thanks{\protect\thanks@warning}%
    \hfill\uppercase{Generalizing Zeckendorf's Theorem}}%
  \let\@mkboth\markboth
  \def\@evenfoot{%
    \thepage \hfil \@volnumber}%
  \def\@oddfoot{%
    \@monthyear \hfil \thepage}%
  }%
\footskip=25pt
\pagestyle{headings}%
}
\makeatother

\theoremstyle{plain}
\numberwithin{equation}{section}
\newtheorem{thm}{Theorem}[section]
\newtheorem{theorem}[thm]{Theorem}
\newtheorem{lemma}[thm]{Lemma}
\newtheorem{corollary}[thm]{Corollary}
\newtheorem{example}[thm]{Example}
\newtheorem{definition}[thm]{Definition}

\newtheorem{proposition}[thm]{Proposition}
\newtheorem{remark}[thm]{Remark}
\newtheorem{conjecture}[thm]{Conjecture}

\newcommand{\ignore}[1]{}

\newcommand\be{\begin{eqnarray}}
\newcommand\ee{\end{eqnarray}}
\newcommand\bea{\begin{eqnarray}}
\newcommand\eea{\end{eqnarray}}
\newcommand\ben{\begin{enumerate}}
\newcommand\een{\end{enumerate}}

\newtheorem{cor}[thm]{Corollary}

\begin{document}
%% replace the values in the next three lines by the correct information

\monthyear{August 2021}
\volnumber{}
\setcounter{page}{1}

\title{Generalizing Zeckendorf's Theorem to Homogeneous Linear Recurrences, II}

\author{Thomas C. Martinez, Steven J. Miller,  Clayton Mizgerd, Jack Murphy, Chenyang Sun}

%\address{\tiny{Department of Mathematics, Carnegie Mellon University, Pittsburgh, PA 15213}}
%\email{rayli1234567890@gmail.com }

\address{\tiny{Department of Mathematics, Harvey Mudd College, Claremont, CA 91711}} \email{tmartinez@hmc.edu}

\address{\tiny{Department of Mathematics and Statistics, Williams College, Williamstown, MA 01267}} \email{sjm1@williams.edu, Steven.Miller.MC.96@aya.yale.edu}
\email{cmm12@williams.edu}
\email{jgm4@williams.edu}
\email{cs19@williams.edu}

%\keywords{Zeckendorf decompositions, Central Limit Theorem, recurrence relations}

\date{\today}

\begin{abstract}  Zeckendorf's theorem states that every positive integer can be written uniquely as the sum of non-consecutive shifted Fibonacci numbers $\{F_n\}$, where we take $F_1=1$ and $F_2=2$. This has been generalized for any Positive Linear Recurrence Sequence (PLRS), which informally is a sequence satisfying a homogeneous linear recurrence with a positive leading coefficient and non-negative integer coefficients. In this and the preceding paper we provide two approaches to investigate linear recurrences with leading coefficient zero, followed by non-negative integer coefficients, with differences between indices relatively prime (abbreviated ZLRR), via two different approaches. The first approach involves generalizing the definition of a legal decomposition for a PLRS found in Kolo\u{g}lu, Kopp, Miller and Wang. We prove that every positive integer $N$ has a legal decomposition for any ZLRR using the greedy algorithm. We also show that a specific family of ZLRRs lost uniqueness of decompositions. The second approach converts a ZLRR to a PLRR that has the same growth rate. We develop the Zeroing Algorithm, a powerful helper tool for analyzing the behavior of linear recurrence sequences. We use it to prove a very general result that guarantees the possibility of conversion between certain recurrences, and develop a method to quickly determine whether a sequence diverges to $+\infty$ or $-\infty$, given any real initial values. This paper investigates the second approach.
\end{abstract}

\thanks{This work was supported by NSF Grants DMS1561945, DMS1659037 and DMS1947438 as well as the Finnerty Fund. We thank the participants of the 2019 and 2020 Williams SMALL REUs and the referee for constructive comments.}

\maketitle

\tableofcontents

%%%%%%%%%%%%%%%%%%%%%%%%%%%%%%%%%%%%%%%%%%%%%%%%%%%%%%%%%%%%%%%%%%%%%%%%%%%%%%%%%%%%%%%%%%%%%%%%%%%%%%%%%%%%%%%%%%%%%%%%%%%%%%%%%%%%
%%%%%%%%%%%%%%%%%%%%%%%%%%%%%%%%%%%%%%%%%%%%%%%%%%%%%%%%%%%%%%%%%%%%%%%%%%%%%%%%%%%%%%%%%%%%%%%%%%%%%%%%%%%%%%%%%%%%%%%%%%%%%%%%%%%%
%%%%%%%%%%%%%%%%%%%%%%%%%%%%%%%%%%%%%%%%%%%%%%%%%%%%%%%%%%%%%%%%%%%%%%%%%%%%%%%%%%%%%%%%%%%%%%%%%%%%%%%%%%%%%%%%%%%%%%%%%%%%%%%%%%%%

%%%%%%%%%%%%%%%%%%%%%%%%%%%%%%%%%%
% INTRODUCTION
%%%%%%%%%%%%%%%%%%%%%%%%%%%%%%%%%%

\section{Introduction and Definitions}\label{sec:intro}

%%%%%%%%%%%%%%%%%%%%%%%%%%%%%%%%%%%%
% PREVIOUS RESULTS AND HISTORY OF PROBLEM
%%%%%%%%%%%%%%%%%%%%%%%%%%%%%%%%%%%%

\subsection{History and Past Results}\label{sec:history}

The Fibonacci numbers are one of the most well-known and well-studied mathematical objects, and have captured the attention of mathematicians since their conception. This paper focuses on a generalization of Zeckendorf's theorem, one of the many interesting properties of the Fibonacci numbers. Zeckendorf \cite{Ze} proved that every positive integer can be written \textbf{uniquely} as the sum of non-consecutive Fibonacci numbers (called the \textit{Zeckendorf Decomposition}), where the (shifted) Fibonacci numbers\footnote{If we use the standard initial conditions then 1 appears twice and uniqueness is lost.} are $F_1 = 1, F_2 = 2, F_3 = 3, F_4 = 5, \dots$. This result has been generalized to other types of recurrence sequences. We set some notation before describing these generalizations.

\begin{definition}[\cite{KKMW}, Definition 1.1, (1)]\label{def:plrrdefinition}
    We say a recurrence relation is a \textbf{Positive Linear Recurrence Relation (PLRR)} if there are non-negative integers $L, c_1, \dots, c_L$ such that
        \begin{equation}
            H_{n+1}\ =\ c_1\, H_n + \cdots + c_L\, H_{n+1-L},
        \end{equation}
        with $L, c_1$ and $c_L$ positive.
\end{definition}

\begin{definition}[\cite{KKMW}, Definition 1.1, (2)]\label{def:plrsdefinition}
    We say a sequence $\{H_n\}_{n=1}^{\infty}$ of positive integers arising from a PLRR is a \textbf{Positive Linear Recurrence Sequence (PLRS)} if $H_1=1$, and for $1 \leq n < L$ we have
        \begin{equation}
            H_{n+1}\ =\ c_1\,H_n + c_2\,H_{n-1} + \cdots + c_n \,H_1 + 1.
        \end{equation}
    We call a decomposition $N=\sum_{i=1}^m a_i H_{m+1-i}$ of a positive integer, and its associated sequence $\{a_i\}_{i=1}^m$, \textbf{legal} if $a_1>0$, the other $a_i\geq 0$, and one of the following holds.\vspace{1mm}
    \begin{itemize}
        \item \emph{Condition 1:} We have $m<L$ and $a_i = c_i$ for $1 \leq i \leq m$,
        \item \emph{Condition 2:} There exists $s\in \{1,\dots,L\}$ such that
        \[
            a_1\ =\ c_1, \ \ a_2\ =\ c_2, \ \ \dots, \ \ a_{s-1}\ =\ c_{s-1}, \ \ a_s\ <\ c_s,
        \]
        $a_{s+1}, \dots , a_{s+\ell} = 0$ for some $\ell \geq 0$, and $\{a_{s+\ell+i}\}_{i=1}^{m-s-\ell}$is legal.
    \end{itemize}
Additionally, we let the empty decomposition be legal for $N=0$.
\end{definition}

Informally, a legal decomposition is one where we cannot use the recurrence relation to replace a linear combination of summands with another summand, and the coefficient of each summand is appropriately bounded; other authors \cite{DG, Ste} use the phrase $G$-ary decomposition for a legal decomposition. For example, if $H_{n+1} = 3H_n + 2H_{n-1} + 4H_{n-2}$, then $H_5 + 3H_4 + 2H_3 + 3H_2$ is legal, while $H_5 + 3H_4 + 2H_3 + 4H_2$ is not (we can replace $3H_4 + 2H_3 + 4H_2$ with $H_5$), nor is $6H_5+2H_4$ (the coefficient of $H_5$ is too large).\\

We now state an important generalization of Zeckendorf's Theorem, and then describe what object we are studying and our results. See \cite{BBGILMT, BM, BCCSW, CFHMN, CFHMNPX, DFFHMPP, Ho,MNPX, MW, Ke, Len} for more on generalized Zeckendorf decompositions, and \cite{GT, MW} for a proof of Theorem \ref{thm:genzeckthmforPLRS}.
\begin{theorem}[Generalized Zeckendorf's theorem for PLRS]\label{thm:genzeckthmforPLRS}
Let $\{H_n\}_{n=1}^{\infty}$ be a \emph{Positive Linear Recurrence Sequence}. Then
\begin{enumerate}
    \item there is a unique legal decomposition for each non-negative integer $N \geq 0$, and
    \item there is a bijection between the set $\mathcal{S}_n$ of integers in $[H_n,H_{n+1})$ and the set $\mathcal{D}_n$ of legal decompositions $\sum_{i\,=\,1}^n a_i\, H_{n+1-i}$.
\end{enumerate}
\end{theorem}

While this result is powerful and generalizes Zeckendorf's theorem to a large class of recurrence sequences, it is restrictive in that the leading term must have a positive coefficient. We examine what happens in general to existence and uniqueness of legal decompositions if $c_1=0$. Some generalizations were studied in \cite{CFHMN, CFHMNPX} on sequences called the $(s,b)$-Generacci sequences. In-depth analysis was done on the $(1,2)$-Generacci sequence, later called the Kentucky sequence,  and the Fibonacci Quilt sequence; the first has uniqueness of decomposition while the second does not.

\begin{definition}\label{def:zlrrdefinition}
    We say a recurrence relation is an $s$-deep \textbf{Zero Linear Recurrence Relation (ZLRR)} if the following properties hold.
    \begin{enumerate}
        \item \emph{Recurrence relation:} There are non-negative integers $s, L, c_1, \dots, c_L$ such that
        \begin{equation}\label{eqn:zlrsrecurrence}
            G_{n+1}\ =\ c_1\, G_n + \cdots + c_s\, G_{n+1-s} + c_{s+1}\, G_{n-s}+ \cdots + c_L\, G_{n+1-L},
        \end{equation}
        with $c_1,\dots, c_s = 0$ and $L, c_{s+1}, c_L$ positive.
        \item \emph{No degenerate sequences:} Let $S =\{m  \mid c_m \neq 0\}$ be the set of indices of positive coefficients. Then $\gcd(S) = 1$.
    \end{enumerate}
\end{definition}

We impose the second restriction to eliminate recurrences with undesirable properties, such as $G_{n+1} = G_{n-1} + G_{n-3}$, where the odd- and even-indexed terms do not interact. Any sequence satisfying this recurrence splits into two separate, independent subsequences. Also note that $0$-deep ZLRRs are just PLRRs whose sequences and decomposition properties are well-understood.\\

A natural question to ask is how decomposition results for PLRSes may be extended to sequences satisfying ZLRRs; we offer two approaches toward addressing it. \cite{MMMS1} focuses on generalizing Zeckendorf's theorem directly to sequences satisfying ZLRSes, while this paper develops a method to convert ZLRRs to PLRRs whose sequences have nice decomposition properties (Theorem \ref{thm:genzeckthmforPLRS}).\\

We develop a powerful helper tool in analyzing linear recurrences, the \textbf{Zeroing Algorithm}; we give a full introduction of how it works in Section \ref{sec:conversion}. It is worth noting that this method has more uses than that of generalizing Zeckendorf's theorem. As the first method required specific initial conditions, converting ZLRRs to PLRRs requires no specificity of initial conditions. We have yet to formally describe a manner to use this method to obtain meaningful results about decompositions, but our hope is that others can use the Zeroing Algorithm to do so. For the rest of this section, for completeness, we review some standard concepts. See for example Section 3 of \cite{Go} and Section 9 of \cite{MT-B}, and for applications, see \cite{BBGILMT}.

\begin{definition}
    Given a recurrence relation
     \begin{equation}\label{eq:recurrence}
            a_{n+1}\ =\ c_1 a_n + \cdots + c_k a_{n+1-k},
        \end{equation}
        we call the polynomial
        \begin{equation}\label{eq:characteristicpolynomial}
             P(x) \ = \  x^k - c_1\, x^{k-1}-c_2\, x^{k-2}-\cdots-c_k
        \end{equation}
        the \textbf{characteristic polynomial} of the recurrence relation. The degree of $P(x)$ is known as the order of the recurrence relation.
\end{definition}

For the rest of the paper we focus on the second approach, converting ZLRR to PLRR. The following definition makes the concept of conversion precise:

\begin{definition}\label{def:derivedfrom}
    We say that a recurrence relation $R_b$ is \textbf{derived from} another recurrence relation $R_a$ if
    \[
        P_b(x) \ = \  P_a(x)Q(x),
    \]
    where $P_a(x)$ and $P_b(x)$ are the characteristic polynomials of $R_a$ and $R_b$ respectively, as defined by equation \eqref{eq:characteristicpolynomial}, and $Q(x)$ is some polynomial with integer coefficients with $Q(x)$ not being the zero polynomial.
\end{definition}

Since the roots of $P_a$ are contained in $P_b$, any sequence satisfying the recurrence relation $R_a$ also satisfies $R_b$, which implies that the two recurrence relations yield the same sequence if the initial values of $\{b_n\}_{n\,=\,1}^{\infty}$ satisfy the recurrence relation $R_a$. This provides motivation for why the idea of a derived PLRR is relevant. To continue, we recall a useful object.\\

\begin{definition}\label{def:principal root}
We call a root $r$ of a polynomial \textbf{principal} if
\begin{enumerate}
    \item it is a positive root of multiplicity 1, and
    \item has magnitude strictly greater than that of any other root.\footnote{Note that, by definition, the principal root is unique.}
\end{enumerate}
\end{definition}
In Lemma \ref{lem:greatestroot}, we prove that the characteristic polynomial of any PLRR or ZLRR has a principal root.\\

\subsection{Main Results}\label{sec:mainresults}

We now state a main result, which has two important corollaries that guarantee the possibility of conversion between certain linear recurrences; the Zeroing Algorithm itself provides a constructive way to do so. We provide some examples of running the Zeroing Algorithm in Appendix \ref{app:examples}.

\begin{theorem}\label{thm:algorithmdivisibility}
Given some PLRR/ZLRR, let $P(x)$ denote its characteristic polynomial, and $r$ its principal root. Suppose we are given an arbitrary sequence of real numbers $\gamma_1,\gamma_2,\dots,\gamma_m$, and define, for $t\le m$,
\begin{equation}
    \Gamma_t(x)\ := \ \gamma_1\,x^{t-1}+\gamma_2\,x^{t-2}+\cdots+\gamma_{t-1}\,x+\gamma_t.
\end{equation}
If $\Gamma_m(r)>0$, there exists a polynomial $p(x)$, divisible by $P(x)$, whose first coefficients are $\gamma_1$ through $\gamma_m$, with no positive coefficients thereafter.
\end{theorem}
\begin{cor}
Given arbitrary integers $\gamma_1$ through $\gamma_m$ with $\Gamma_m(r)>0$, there is a recurrence derived from $P(x)$ which has first coefficients $\gamma_1$ through $\gamma_m$ with no negative coefficients thereafter.
\end{cor}
\begin{cor}\label{cor: conversionBIG}
Every ZLRR has a derived PLRR.
\end{cor}

We list some examples of ZLRRs with the derived PLRRs that are found with the Zeroing Algorithm in Appendix \ref{app:list}.\\

A natural question of interest that arises in the study of recurrences is the behavior of the size of terms in a recurrence sequence. The Fibonacci sequence behaves like a geometric sequence whose ratio is the golden ratio, and there is an analogous result for general linear recurrence sequences, proven in \cite{BBGILMT}.

\begin{theorem}\label{thm:binetexpansion}
Let $P(x)$ be the characteristic polynomial of some linear recurrence relation, and let the roots of $P(x)$ be denoted as $r_1,r_2,\dots,r_j$, with multiplicities $m_1,m_2,\dots,m_j\ge1$, respectively. \\

Consider a sequence $\{a_n\}_{n=1}^{\infty}$ of complex numbers satisfying the recurrence relation. Then there exist polynomials $q_1,q_2,\dots,q_j$, with $\deg(q_i)\le m_i-1$, such that
\begin{equation}\label{eq:binetexpansion}
    a_n\ = \ q_1(n)\,r_1^n+q_2(n)\,r_2^n+\cdots+q_j(n)\,r_j^n.
\end{equation}
\end{theorem}
\begin{definition}
We call \eqref{eq:binetexpansion} the \textbf{Binet expansion} of the sequence $\{a_n\}_{n\,=\,1}^{\infty}$, in analogy to the Binet Formula that provides a closed form for Fibonacci numbers.
\end{definition}

If given a PLRR/ZLRR with some real initial values, one might ask whether the terms eventually diverge to positive infinity or negative infinity. \footnote{Note that we allow the initial values to be arbitrary real numbers, which would result in the sequence taking on one of three behaviors: diverging to $+\infty$, diverging to $-\infty$, or oscillating in sign and having magnitude $o(r^n)$.} One approach is to compute as many terms as needed for the eventual behavior to emerge; unfortunately, this could be very time-consuming. One could alternately solve for the Binet expansion, which often requires an excessive amount of computation. \\

The fact that the characteristic polynomials for PLRR/ZLRRs have a principal root $r$ allows for a shortcut. Consider the Binet expansion of a ZLRS/PLRS; the coefficient attached to the $r^n$ term, whenever nonzero, indicates the direction of divergence. We develop the following method to determine the sign of this coefficient from the initial values of the recurrence sequence.

\begin{theorem}\label{thm:algorithm determination}
Given a ZLRS/PLRS $\{a_n\}_{n\,=\,1}^{\infty}$ with characteristic polynomial $P(x)$ and real initial values $a_1,a_2,\dots,a_k$, consider the Binet expansion of $\{a_n\}_{n\,=\,1}^{\infty}$. The sign of the coefficient attached to $r^n$ equals the sign of
\begin{equation}
    Q(x)\ := \ a_1\,x^{k-1}+(a_2-d_2)\,x^{k-2}+(a_3-d_3)\,x^{k-3}+\cdots+(a_k-d_k)
\end{equation}
evaluated at $x=r$, where
\begin{equation}
    d_i\ = \ a_1\,c_{i-1}+a_2\,c_{i-2}+\cdots+a_{i-1}\,c_1\ = \ \sum_{j\,=\,1}^{i-1}\,a_j\,c_{i-j}.
\end{equation}
\end{theorem}

%%What happens when the Zeroing Algorithm Terminates? What does it output? Why is it important?
\par In Section \ref{sec:runtime} we investigate the run-time of the Zeroing Algorithm, and discover that the run-time depends exclusively on the initial configuration of the algorithm. Particularly, we show that ZLRRs with principal root closest to $1$ will take the longest to be converted into a derived PLRR by the Zeroing Algorithm. We conclude in Section \ref{sec:conclusion} with some open questions for future research.

%%%%%%%%%%%%%%%%%%%%%%%%%%%%%%%%%%%%%%%%%%%%%%%%%%%%%%%%%%%%%%%%%%%%%%%%%%%%%%%%%%%%%%%%%%%%%%%%%%%%%%%%%%%%%%%%%%%%%%%%%%%%%%%%%%%%%%%%%%%%%%%%%%%%%%
%%%%%%%%%%%%%%%%%%%%%%%%%%%%%%%%%%%%%%%%%%%%%%%%%%%%%%%%%%%%%%%%%%%%%%%%%%%%%%%%%%%%%%%%%%%%%%%%%%%%%%%%%%%%%%%%%%%%%%%%%%%%%%%%%%%%%%%%%%%%%%%%%%%%%%
%%%%%%%%%%%%%%%%%%%%%%%%%%%%%%%%%%%%%%%%%%%%%%%%%%%%%%%%%%%%%%%%%%%%%%%%%%%%%%%%%%%%%%%%%%%%%%%%%%%%%%%%%%%%%%%%%%%%%%%%%%%%%%%%%%%%%%%%%%%%%%%%%%%%%%

\section{Eventual Behavior of Linear Recurrence Sequences}\label{sec:tools}

In this section, we introduce important lemmas related to the roots of characteristic polynomials. In the celebrated Binet's Formula for Fibonacci numbers, the principal root of its characteristic polynomial (i.e., the golden ratio) determines the behavior of the sequence as nearly geometric, with the golden ratio being the common ratio. We extend this characterization of near-geometric behavior to more general linear recurrences.

%%%%%%%%%%%%%%%%%%%%%%%%%%%%%%%%%%%%%%%%%%%%%%%%%%%%%%%%%%%%%%%%%%%%%%%%%
%%%%%%%%%%%%%%%%%%%%%%%%%%%%%%%%%%%%%%%%%%%%%%%%%%%%%%%%%%%%%%%%%%%%%%%%%
%%%%%%%%%%%%%%%%%%%%%%%%%%%%%%%%%%%%%%%%%%%%%%%%%%%%%%%%%%%%%%%%%%%%%%%%%
\subsection{Properties of Characteristic Polynomials}

Lemma \ref{lem:monotonicallyincreasinglessstrong} characterizes the limiting behavior of recurrence relations of the form \eqref{eq:recurrence}, with $c_i$ non-negative integers for $1\leq i\leq k$ and $c_k>0$. We first state important consequences of the definition of the principal root.

\begin{lemma}\label{lem:greatestroot}
Consider $P(x)$ as in \eqref{eq:characteristicpolynomial} and let $S:=\{m \, \mid \, c_m \ne 0\}$. Then
\begin{enumerate}
    \item there exists exactly one positive root $r$, and this root has multiplicity $1$,
    \item every root $z \in \mathbb{C}$ satisfies $|z|\le r$, and
    \item if $\gcd(S) = 1$,\footnote{Note that this is Condition 2 from Definition \ref{def:zlrrdefinition}, and thus met by all $s$-deep ZLRRs.} then $r$ is the unique root of greatest magnitude.
\end{enumerate}
\end{lemma}

\begin{proof}
By Descartes's Rule of Signs, $P(x)$ has exactly one positive root of multiplicity one, completing the proof of Part (1).\\

Now, consider any root $z\in\mathbb{C}$ of $P(x)$; we have $z^k=c_1z^{k-1}+c_2z^{k-2}+\cdots+c_k$. Taking the magnitude, we have
\begin{align}
|z|^k\ = \ |z^k|\ = \ |c_1z^{k-1}+c_2z^{k-2}+\cdots+c_k|&\ \le\ |c_1z^{k-1}|+|c_2z^{k-2}|+\cdots+|c_k|\nonumber\\
&\ = \ c_1|z|^{k-1}+c_2|z|^{k-2}+\cdots+c_k,
\end{align}
which implies $P(|z|)\le0$. Since $P(x)$ becomes arbitrarily large with large values of $x$, we see that there is a positive root at or above $|z|$ by the Intermediate Value Theorem, which completes the proof of Part (2). \\

Finally, suppose $\gcd(S) = 1$. Suppose to the contrary that a non-positive root $z$ satisfies $|z| = r$; we must have $P(|z|)=0$, which implies
\begin{equation}
    |z^k|\ = \ |c_1\,z^{k-1}+c_2\,z^{k-2}+\cdots+c_k|\ = \ |c_1\,z^{k-1}|+|c_2\,z^{k-2}|+\cdots+|c_k|.
\end{equation}
This equality holds only if the complex numbers $c_1\,z^{k-1},c_2\,z^{k-2},\dots,c_k$ share the same argument; since $c_k>0$, $z^{k-j}$ must be positive for all $c_j\ne0$. This implies $z^k$, as a sum of positive numbers, is positive as well. Writing $z=|z|\,e^{i\theta}$, we see that the positivity of $z^k=|z|^k\,e^{ik\theta}$ implies $k\theta$ is a multiple of $2\pi$, and consequently, $\theta = 2\pi d / k$ for some integer $d$. We may reduce this to $2\pi d' / k'$ for relatively prime $d',k'$.\\

Let $J:=S\cup\{0\}$. Since $z^{k-j}$ is positive for all $j\in J$, we see that $2\pi d'\,(k-j)/k'$ is an integer multiple of $2\pi$, so $k'$ divides $d'\,(k-j)$; as $d'$ and $k'$ are relatively prime we have $k'$ divides $k-j$. Since the elements of $J$ have greatest common divisor 1, so do\footnote{Observe that $k$ is in both $J$ and $K$. Suppose to the contrary that some $q>1$ divides every element of $K$; then, every element of $\{k-\kappa\mid\kappa\in K\}=J$ is divisible by $q$, which is impossible.}
the elements of $K:=\{k-j\mid j\in J\}$. Since $k'$ divides every element of $K$, we must have $k'=1$, so $\theta=2\pi d'$ and thus $z$ is a positive root. This is a contradiction, completing the proof of Part (3).
\end{proof}

Next, we state a lemma that sheds light on the growth rate of the terms of a ZLRR/PLRR with a specific set of initial values.

\begin{lemma}\label{lem:monotonicallyincreasinglessstrong}
For a PLRR/ZLRR, let $r$ be the principal root of its characteristic polynomial $P(x)$. Then, given initial values $a_i=0$ for $0\le i\le k-2$, $a_{k-1}=1$, we have
\begin{equation}
    \lim_{n\to\infty} \frac{a_n}{r^n} \ = \  C,
\end{equation}
where $C > 0$. Furthermore, the sequence $\{a_n\}_{n\,=\,1}^{\infty}$ is eventually monotonically increasing.
\end{lemma}
\begin{proof}
Since $r$ has multiplicity $1$, $q_1$ is a constant polynomial. To see geometric behavior, we note that
\begin{equation}
    \lim_{n\to\infty}\frac{a_n}{r^n}\ = \ \lim_{n\to\infty}q_1(n)\,\left(\frac{r^n}{r^n}\right)+\lim_{n\to\infty}q_2(n)\,\left(\frac{r_2}{r}\right)^n+\cdots+\lim_{n\to\infty}q_j(n)\,\left(\frac{r_j}{r}\right)^n.
\end{equation}
Since $|r|\,>|r_i|$ for all $2\le i\le j$, each limit with a $(r_i/r)^n$ term disappears, leaving just $q_1$, which must be positive, since the sequence $a_n$ does not admit negative terms.\\

To see that $a_n$ is eventually increasing, consider the sequence
\begin{align}
    A_n&\ := \ a_{n+1}-a_n\nonumber\\
    &\ \  = \ (q_1r_1-q_1)\,r_1^n+\left(q_2(n+1)\,r_2-q_2(n)\right)\,r_2^n+\cdots+\left(q_j(n+1)\,r_j-q_j(n)\right)\,r_j^n.
\end{align}

A similar analysis shows
\begin{equation}
    \lim_{n\,\to\,\infty}\frac{\left(q_2(n+1)\,r_2-q_2(n)\right)\,r_2^n+\cdots+\left(q_j(n+1)\,r_j-q_j(n)\right)\,r_j^n}{\left(q_1\,r_1-q_1\right)\,r_1^n}\ = \ 0,
\end{equation}
implying that the term $(q_1r_1-q_1)\,r_1^n$ grows faster than the sum of the other terms; thus $A_n$ is eventually positive as desired.
\end{proof}

\begin{corollary}\label{cor:monotonicallyincreasing}
For a PLRR/ZLRR, let $r$ be the principal root of its characteristic polynomial $P(x)$. Then, given initial values satisfying $a_i\geq0$ for $0\le i\le k-1$ and $a_i > 0$ for some $0\le i \le k-1$, we have
\begin{equation}
    \lim_{n\,\to\,\infty} \frac{a_n}{r^n} \ = \  C,
\end{equation}
where $C > 0$. Furthermore, the sequence $\{a_n\}_{n\,=\,1}^{\infty}$ is eventually monotonically increasing. That is, Lemma \ref{lem:monotonicallyincreasinglessstrong} extends to any set of non-negative initial values that are not all zero.
\end{corollary}
\begin{proof}
We first note that the derivation of \eqref{eq:binetexpansion} does not rely on the initial values; any sequence satisfying the recurrence takes on this form.\\

Since one of the initial values $a_0,a_1,\dots,a_{k-1}$ is a positive integer, we know that one of $a_k,a_{k+1},\dots,a_{2k-1}$ is also a positive integer by the recurrence relation, which forces $a_{n}$ to be at least $a_{n-k}$. Define $i\in[k,2k-1]$ to be an index such that $a_i$ is positive. Consider the sequence $b_n=a_{n+i-k+1}$, which has $b_{k-1}=a_i>0$. By the recurrence relation, we have $b_n\ge a_n$ for all $n$, which would be impossible if the Binet expansion of $b_n$ had a non-positive coefficient attached to the $r^n$ term. Eventual monotonicity thus follows.
\end{proof}

\subsection{A Generalization of Binet's Formula}

In general, the Binet expansion of a recurrence sequence is quite unpleasant to compute or work with. However, things become much simpler when the characteristic polynomial has no multiple roots. In that case, we may construct an explicit formula for the $n$th term of the sequence, given a nice set of initial values. Keeping in mind that linear combinations of sequences satisfying a recurrence also satisfy the recurrence, one could construct a formula for the $n$th term given arbitrary initial values.

\begin{theorem}\label{thm: Binet Generalization}
Consider a ZLRR with characteristic polynomial $P(x)$ that does not have multiple roots, and initial values $a_i=0$ for $0\le i\le k-2$, $a_{k-1}=1$. Then each term of the resulting sequence may be expressed as
\begin{equation}
    a_n\ = \ c_1\,r_1^n+c_2\,r_2^n+\cdots+c_k\,r_k^n,
\end{equation}
where the $r_i$ are the distinct roots of $P(x)$, and $c_i=1/P'(r_i)$.
\end{theorem}

Before providing a proof of Theorem \ref{thm: Binet Generalization}, we illustrate with a motivating example: Binet's Formula.

\begin{example}
Consider the Fibonacci Numbers with $F_0=0,F_1=1$. Let $P(x)=x^2-x-1$, which has roots $\alpha = (1+\sqrt{5})/2$ and $\beta=(1-\sqrt5)/2$. Then $P'(x) = 2x-1$ and it is easy to verify that $1/P'(\alpha) = 1/\sqrt{5}$ and $1/P'(\beta) = - 1/\sqrt{5}$, leading to the well known Binet formula for the Fibonacci numbers.
\end{example}

We now prove Theorem \ref{thm: Binet Generalization}.

\begin{proof}
Since each root has multiplicity $1$, the existence of such explicit form follows from the Binet expansion (see Theorem \ref{thm:binetexpansion}), so we are left to prove that $c_i = 1/P'(r_i)$. Using the initial values, we see that the $c_i$ are solutions to the linear system
\begin{equation}
\begin{pmatrix}
1 & 1 & 1 & \cdots & 1\\
r_1 & r_2 & r_3 & \cdots & r_k\\
r_1^2 & r_2^2 & r_3^2 & \cdots & r_k^2\\
\vdots & \vdots & \vdots & \ddots & \vdots\\
r_1^{k-1} & r_2^{k-1} & r_3^{k-1} & \cdots & r_k^{k-1}
\end{pmatrix}
\begin{pmatrix}
c_1\\
c_2\\
c_3\\
\vdots\\
c_k
\end{pmatrix}\ = \
\begin{pmatrix}
0\\
0\\
0\\
\vdots\\
1
\end{pmatrix}.
\end{equation}

Denote the matrix above by $A$, let $A_i$ be the matrix formed by replacing column $i$ of $A$ with the column vector of zeroes and a single 1 in the bottom-most index, and let $M_{ki}$ be the $k,i$ minor matrix of $A$ formed by deleting row $k$ and column $i$. A standard application of Cramer's rule yields:
\begin{align}
    c_i&\ = \ (-1)^{k+i}\left(\prod_{\substack{1\,\le\, a\,<\,b\,\le \,k\\ a,\,b\,\ne\, i}}(r_b-r_a)\right)\bigg/\left(\prod_{1\,\le\, a\,<\,b\, \le\, k}(r_b-r_a)\right)\nonumber\\
    &\ = \ (-1)^{k+i}\bigg/\left(\prod_{\substack{1\,\le\, a\,<\,b\,\le\, k\\ a\, =\, i\text{ or }b\,=\,i}}(r_b-r_a)\right)\nonumber\\
    &\ = \ \frac{(-1)^{k+i}}{(r_i-r_1)(r_i-r_2)\cdots(r_i-r_{i-1})(r_{i+1}-r_i)\cdots(r_{k-1}-r_i)(r_k-r_i)}\nonumber\\
    &\ = \ \frac{(-1)^{k+i}}{\left(\prod_{j\,=\,1}^{i-1}(r_i-r_j)\right)(-1)^{k-i}\left(\prod_{j\,=\,i+1}^{k}(r_i-r_j)\right)}\nonumber\\
    &\ = \ 1\bigg/\prod_{\substack{1\,\le \,j\,\le\, k\\ j\,\ne\, i}}(r_i-r_j).
\end{align}
Note that the product is simply the function
\begin{equation}
    f(x)\ = \ \prod_{\substack{1\,\le j\,\le \,k\\ j\,\ne\, i}}(x-r_j)
\end{equation}
evaluated at $x=r_i$. To evaluate this, we may rewrite
\begin{align}
    f(r_i)&\ = \ \lim_{x\, \to\ r_i}f(x)\ = \ \lim_{x\, \to\ r_i}\prod_{\substack{1\,\le\, j\,\le\, k\\ j\,\ne\, i}}(x-r_j)\nonumber\\
    &\ = \ \lim_{x\, \to\ r_i}\frac{(x-r_i)}{(x-r_i)}\prod_{\substack{1\,\le \,j\,\le\, k\\ j\,\ne\, i}}(x-r_j)\ = \ \lim_{x\,\to\ r_i}\frac{\prod_{1\,\le\, j\,\le\, k}(x-r_j)}{x-r_i}\nonumber\\
    &\ = \ \lim_{x\, \to\ r_i}\frac{P(x)}{x-r_i},
\end{align}
which equals $P'(r_i)$ by L'H\^opital's rule. We thus have $c_i = 1/f(r_i) = 1/P'(r_i)$, completing the proof.
\end{proof}

%%%%%%%%%%%%%%%%%%%%%%%%%%%%%%%%%%%%%%%%%%%%%%%%%%%%%%%%%%%%%%%%%%%%%%%%%
%%%%%%%%%%%%%%%%%%%%%%%%%%%%%%%%%%%%%%%%%%%%%%%%%%%%%%%%%%%%%%%%%%%%%%%%%
%%%%%%%%%%%%%%%%%%%%%%%%%%%%%%%%%%%%%%%%%%%%%%%%%%%%%%%%%%%%%%%%%%%%%%%%%

%%%%%%%%%%%%%%%%%%%%%%%%%%%%%%%%%%%%%%%%%%%%%%%%%%%%%%%%%%%%%%%%%%%%%%%%%%%%%%%%%%%%%%%%%%%%%%%%%%%%%%%%%%%%%%%%%%%%%%%%%%%%%%%%%%%%%%%%%%%%%%%%%%%%%%
%%%%%%%%%%%%%%%%%%%%%%%%%%%%%%%%%%%%%%%%%%%%%%%%%%%%%%%%%%%%%%%%%%%%%%%%%%%%%%%%%%%%%%%%%%%%%%%%%%%%%%%%%%%%%%%%%%%%%%%%%%%%%%%%%%%%%%%%%%%%%%%%%%%%%%
%%%%%%%%%%%%%%%%%%%%%%%%%%%%%%%%%%%%%%%%%%%%%%%%%%%%%%%%%%%%%%%%%%%%%%%%%%%%%%%%%%%%%%%%%%%%%%%%%%%%%%%%%%%%%%%%%%%%%%%%%%%%%%%%%%%%%%%%%%%%%%%%%%%%%%
\section{The Zeroing Algorithm and Applications}\label{sec:conversion}

An alternate approach to understanding decompositions arising from ZLRRs is to see if for every ZLRR one could associate a PLRR with similar behavior: a derived PLRR. In this section, we develop the machinery of the Zeroing Algorithm, which is an extremely powerful tool for understanding recurrence sequences analytically. We prove a very general result about derived recurrences that implies every ZLRS has a derived PLRS.

%%%%%%%%%%%%%%%%%%%%%%%%%%%%%%%%%%%%%%%%%%%%%%%%%%%%%%%%%%%%%%%%%%%%%%%%%
%%%%%%%%%%%%%%%%%%%%%%%%%%%%%%%%%%%%%%%%%%%%%%%%%%%%%%%%%%%%%%%%%%%%%%%%%
%%%%%%%%%%%%%%%%%%%%%%%%%%%%%%%%%%%%%%%%%%%%%%%%%%%%%%%%%%%%%%%%%%%%%%%%%
\subsection{The Zeroing Algorithm}

Consider some ZLRS/PLRS with characteristic polynomial
\begin{equation}\label{eq:P(x)}
     P(x)\ := \ x^k-c_1\,x^{k-1}-c_2\,x^{k-1}-\cdots-c_k,
\end{equation}
and choose a sequence of $k$ real numbers $\beta_1,\beta_2,\dots,\beta_{k}$; the $\beta_i$ are considered the input of the algorithm. For nontriviality, the $\beta_i$ are not all zero. We define the \textit{Zeroing Algorithm} to be the following procedure. First, create the polynomial
\begin{equation}\label{eq:zeroalginitial}
    Q_0(x)\ := \ \beta_1\,x^{k-1}+\beta_2\,x^{k-2}+\cdots+\beta_{k-1}\,x+\beta_k.
\end{equation}
Next, for $t\ge1$, define a sequence of polynomials
\begin{equation}\label{eq: polyrecur}
    Q_{t}(x)\ := \ x\,Q_{t-1}(x)-q(1,t-1)\,P(x),
\end{equation}
indexed by $t$, where $q(1,t)$ is the coefficient of $x^{k-1}$ in $Q_t(x)$. We \textit{terminate} the algorithm at step $t$ if $Q_t(x)$ does not have positive coefficients. An example run of the Zeroing Algorithm is provided in Appendix \ref{app:examples}.\\

To understand the algorithm through linear recurrences, we denote by $q(n,t)$ the coefficient of $x^{k-n}$ in $Q_t(x)$, where $n$ ranges from $1$ to $k$.
The recurrence relation on the polynomials \eqref{eq: polyrecur} yields the following system of recurrence relations
\begin{align}
    q(1,t)&\ = \ q(2,t-1)+c_1\,q(1,t-1),\label{eq: coeffrecur}\\
    q(2,t)&\ = \ q(3,t-1)+c_2\,q(1,t-1),\nonumber \\
    &\ \ \ \vdots \ \nonumber\\
    q(k-1,t)&\ = \ q(k,t-1)+c_{k-1}\,q(1,t-1), \nonumber\\
    q(k,t)&\ = \ c_k\,q(1,t-1),\nonumber
\end{align}
with initial values
\begin{align*}
    q(1,0)\ = \ \beta_1, \ \ \ \  q(2,0)\ = \ \beta_2,\ \ \ \  \cdots, \ \ \ \  q(k,0)\ = \ \beta_k.
\end{align*}
Note that if $q(1,t)$ through $q(k,t)$ are all non-positive, then so are $q(1,t+1)$ through $q(k,t+1)$; the same holds for nonnegativity.

\begin{lemma}\label{lem:satisfiesRecurrence}
The sequence $q(1,t)$ satisfies the recurrence specified by the characteristic polynomial $P(x)$. For each $1\le n\le k$, $q(n,t)$ is a positive linear combination of $q(1,t-j)$, $j=1,...,i+1$:
\begin{align}
    q(n,t)&\ = \ c_n\,q(1,t-1)+c_{n+1}\,q(1,t-2)+\cdots+c_k\,q(1,t-(k+1-n))\nonumber\\
    &\ = \ \sum_{i=0}^{k-n}\,c_{n+i}\,q(1,t-(i+1)).
\end{align}
\end{lemma}

\begin{proof}
 We first examine the sequence $q(1,t)$. For $t\ge k$, we have
 \begin{align}
     q(1,t)&\ = \ c_1\,q(1,t-1)+q(2,t-1)\nonumber\\
     &\ = \ c_1\,q(1,t-1)+c_2\,q(1,t-2)+q(3,t-2)\nonumber\\
     &\ \ \ \vdots\ \nonumber \\
     &\ = \ c_1\,q(1,t-1)+c_2\,q(1,t-2)+\cdots+c_{k-1}\,q(1,t-(k-1))+q(k,t-(k-1))\nonumber\\
     &\ = \ c_1\,q(1,t-1)+c_2\,q(1,t-2)+\cdots+c_{k-1}\,q(1,t-(k-1))+c_k\,q(1,t-k),
 \end{align}
 which is what we want.\\

The proof for $q(n,t)$ is similar:
 \begin{align}
     q(n,t)&\ = \ c_n\,q(1,t-1)+q(n+1,t-1)\nonumber\\
     &\ = \ c_n\,q(1,t-1)+c_{n+1}\,q(1,t-2)+q(n+2,t-2)\nonumber\\
     &\ = \ c_n\,q(1,t-1)+c_{n+1}\,q(1,t-2)+c_{n+2}\,q(1,t-3)+q(n+3,t-3)\nonumber\\
     &\ \ \ \vdots\ \nonumber\\
     &\ = \ c_n\,q(1,t-1)+c_{n+1}\,q(1,t-2)+\cdots+q(n+(k-n),t-(k-n))\nonumber\\
     &\ = \ c_n\,q(1,t-1)+c_{n+1}\,q(1,t-2)+\cdots+q(k,t-(k-n))\nonumber\\
     &\ = \ c_n\,q(1,t-1)+c_{n+1}\,q(1,t-2)+\cdots+c_k\,q(1,t-(k-n+1)),
 \end{align}
 as desired.
\end{proof}

Now we may prove a very useful result.

\begin{lemma}
Let $r$ be the principal root of $P(x)$. Consider the Binet expansion of the sequence $q(n,t)$ (indexed by $t$) for each $n$. The sign of the coefficient attached to the term $r^t$ equals the sign of $Q_0(r)$.
\end{lemma}

\begin{proof}
 Recall the recurrence relation $Q_{t}(x)=x\,Q_{t-1}(x)-q(1,t-1)\,P(x)$. Evaluating at $x=r$, the $P(x)$ term drops out and we have $Q_{t}(r)=r\,Q_{t-1}(r)$. Iterating this procedure gives $r^t\,Q_0(r)$.

Recalling that $q(n,t)$ is defined to be the coefficient of $x^{k-n}$ in $Q_t(x)$, we have
 \begin{equation}
     r^t\,Q_0(r)\ = \ Q_t(r)\ = \ r^{k-1}\,q(1,t)+r^{k-2}\,q(2,t)+\cdots+r\,q(k-1,t)+q(k,t).
 \end{equation}
 Note that this implies that the sequence $Q_t(r)$ satisfies the recurrence specified by $P(x)$ as well. Since each $q(n,t)$ is a positive linear combination of $q(1,t-j)$, $j=1,...,i+1$, they all have the same sign on the coefficient of the $r^t$ term in their explicit expansion as a sum of geometric sequences, and this sign equals the sign of the coefficient of $r^t$ in the expansion of $Q_t(r)$. It remains to show the sign in $Q_t(r)$ equals the sign of $Q_0(r)$.\\

 Consider the quantity $\lim_{t\to\infty} Q_t(r) / r^t$, which extracts the coefficient of the $r^t$ term in $Q_t(r)$. Since $Q_t(r)=r^tQ_0(r)$, we have
 \begin{equation}
     \lim_{t\to\infty}\,\frac{Q_t(r)}{r^t}\ = \ \lim_{t\to\infty}\,\frac{r^t\,Q_0(r)}{r^t}\ = \ Q_0(r)
 \end{equation}
as desired.
\end{proof}

We can now establish an exact condition on when the Zeroing Algorithm terminates.

\begin{theorem}\label{thm:algorithmtermination}
Let $Q_0(x)$ be as defined in \eqref{eq:zeroalginitial} and let $r$ be the principal root of $P(x)$ . The Zeroing Algorithm terminates if and only if $Q_0(r)<0$.
\end{theorem}

\begin{proof}[Proof of Theorem \ref{thm:algorithmtermination}]
If $Q_0(r)<0$, then the coefficient of $r^t$ in the expansion of $q(n,t)$ is also negative for each $n$; this implies $q(n,t)$ diverges to negative infinity, and that there must be some $t$ when $q(n,t)$ is non-positive for each $n$.\\

For the other direction, if $Q_0(r)\ge0$ then suppose to the contrary that there is some $t_0$ where $q(n,t_0)\le0$ for all $n$. Then we would have
\begin{equation}
    r^{t_0}\,Q_0(r)\ = \ Q_{t_0}(r)\ = \ r^{k-1}\,q(1,t_0)+r^{k-2}\,q(2,t_0)+\cdots+r\,q(k-1,t_0)+q(k,t_0)\ \le \ 0,
\end{equation}
which implies $Q_0(r)\le0$, forcing $Q_0(r)=0$.\\

Notice that this equality only occurs when $q(1,t_0)=q(2,t_0)=\cdots=q(k,t_0)=0$. This implies for each $n$, $q(n,t)=0$ for all $t>t_0$, so each $q(n,t)$ is identically zero, which contradicts our assumption of non-triviality.
\end{proof}

%%%%%%%%%%%%%%%%%%%%%%%%%%%%%%%%%%%%%%%%%%%%%%%%%%%%%%%%%%%%%%%%%%%%%%%%%
%%%%%%%%%%%%%%%%%%%%%%%%%%%%%%%%%%%%%%%%%%%%%%%%%%%%%%%%%%%%%%%%%%%%%%%%%
%%%%%%%%%%%%%%%%%%%%%%%%%%%%%%%%%%%%%%%%%%%%%%%%%%%%%%%%%%%%%%%%%%%%%%%%%
\subsection{A General Conversion Result}

Now that we have developed the main machinery of the Zeroing Algorithm, we can prove a very general result on converting between linear recurrences.

\begin{proof}[Proof of Theorem \ref{thm:algorithmdivisibility}]
For ease of notation, extend the $\gamma$ sequence by setting $\gamma_i=0$ for $i>m$. We modify the Zeroing Algorithm slightly to produce the desired $p(x)$.

Consider a sequence of polynomials $Q_t(x)$ of degree at most $k-1$, with
\begin{align}
    Q_1(x)&\ = \ \gamma_1\,(P(x)-x^k),\nonumber\\
    Q_t(x)&\ = \ x\,Q_{t-1}(x)-(q(1,t-1)-\gamma_t)\,P(x)-\gamma_t\,x^k,
\end{align}
where again, $q(n,t)$ denotes the coefficient of $x^{k-n}$ in $Q_t(x)$. Note that after iteration $m$, $\gamma_t=0$ and we have the unmodified Zeroing Algorithm again.

\begin{lemma} \label{lem:modifiedAlgorithm}
Define $p_t(x):=x^{k}\,\Gamma_t(x)+Q_t(x)$. At each iteration $t$, we have the following:
\begin{enumerate}[i)]
\item $P(x)$ divides $p_t(x)$,
\item the first $t$ coefficients of $p_t(x)$ are $\gamma_1$ through $\gamma_t$, and
\item $Q_t(r)\,=\,-r^k\,\Gamma_t(r)$.
\end{enumerate}
\end{lemma}
\begin{proof}
A straightforward induction argument suffices for all of them.
\ \\
i) We have
\begin{equation}
p_1(x)\ = \ x^k\,\gamma_1(x)+Q_1(x)\ = \ x^k\,\gamma_1+\gamma_1\,(P(x)-x^k)\ = \ \gamma_1\,P(x).
\end{equation}
Assuming $P(x)$ divides $p_t(x)$, we have
\begin{align}
    p_{t+1}(x)&\ = \ x^k\,\Gamma_{t+1}(x)+Q_{t+1}(x)\nonumber\\
    &\ = \ x^{k}\,(\gamma_1\,x^{t}+\gamma_2\,x^{t-1}+\cdots+\gamma_{t+1})+Q_{t+1}(x)\nonumber\\
    &\ = \ x\cdot x^{k}\,(\gamma_1\,x^{t-1}+\gamma_2\,x^{t-2}+\cdots+\gamma_t)+\gamma_{t+1}\,x^k+x\,Q_{t}(x) \nonumber\\
    & \ \ \ \ \ - \ (q(1,t)-\gamma_{t+1})\,P(x)-\gamma_{t+1}\,x^k\nonumber\\
    &\ = \ x\,(x^{k}\,(\,\gamma_1\,x^{t-1}+\gamma_2\,x^{t-2}+\cdots+\gamma_t)+Q_t(x))-(q(1,t)-\gamma_{t+1})\,P(x)\nonumber\\
    &\ = \ x\,p_t(x)-(q(1,t)-\gamma_{t+1})\,P(x),
\end{align}
which is divisible by $P(x)$ by the inductive hypothesis.

\ \\
ii) We first prove that $Q_t(x)$ has degree at most $k-1$. This is certainly true for $Q_1(x)=\gamma_1(P(x)-x^k)$. Assume $Q_t(x)$ as degree at most $k-1$; we then have
\begin{equation}
    Q_{t+1}(x)\ = \ x\,Q_{t}(x)-(q(1,t)-\gamma_{t+1})\,P(x)-\gamma_{t+1}\,x^k.
\end{equation}
It is evident that the highest power of $x$ to appear is $x^k$, which has coefficient
\begin{equation}
q(1,t)-(q(1,t)-\gamma_{t+1})-\gamma_{t+1}\ = \ 0.
\end{equation}
From the construction $p_t(x):=x^{k}\,\Gamma_t(x)+Q_t(x)$. It is evident that the first $t$ coefficients are just those of $\Gamma_t(x)$.

\ \\
iii) We have
\begin{equation}
    Q_1(r)\ = \ \gamma_1\,(P(r)-r^k)\ = \ -r^k\,\gamma_1.
\end{equation}
Suppose $Q_t(r)=-r^k\,\Gamma_t(r)$; we have
\begin{align}
    Q_{t+1}(r)&\ = \ r\,Q_{t}(r)-(q(1,t)-\gamma_{t+1})\,P(r)-\gamma_{t+1}\,r^k\nonumber\\
    &\ = \ r\,(-r^k\,\Gamma_t(r))-\gamma_{t+1}\,r^k\nonumber\\
    &\ = \ -r^k\,(r\,\Gamma_t(r)+\gamma_{t+1})\nonumber\\
    &\ = \ -r^k\,\Gamma_{t+1}(r).
\end{align}
\end{proof}

Now we have $Q_m(r)=-r^m\,\Gamma_m(r)<0$, since $\Gamma_m(r)>0$. Running the Zeroing Algorithm starting with $Q_m(x)$ yields some $Q_{m+t_0}(x)$ that does not have positive coefficients. We see that since $p_{m+t_0}(x)=x^k\,\Gamma_{m+t_0}(x)+Q_{m+t_0}(x)$ is divisible by $P(x)$, has its initial $m+t_0$ coefficients as $\gamma_1$ through $\gamma_m$ followed by $t_0$ 0's, and thus does not have positive coefficients after $\gamma_m$; we may choose $p(x)=p_{m+t_0}(x)$.
\end{proof}
\begin{cor}
Given $\gamma_1=1$ and arbitrary integers $\gamma_2$ through $\gamma_m$ with $\Gamma_m(r)>0$, there is a recurrence derived from $P(x)$ whose characteristic polynomial has its first coefficients as $\gamma_1$ through $\gamma_m$ with no positive coefficients thereafter.
\end{cor}
\begin{proof}
Consider $p(x)$ from Theorem \ref{thm:algorithmdivisibility}, whose first coefficients are $\gamma_1$ through $\gamma_m$. Since $\gamma_1=1$, $p(x)$ is the characteristic polynomial of a linear recurrence. In fact, since $\gamma_2$ through $\gamma_m$ are integers, $p(x)$, and thus the recurrence, has integer coefficients.
\end{proof}
\begin{cor}\label{cor: conversion}
Every ZLRR has a derived PLRR.
\end{cor}
\begin{proof}
Take $m=2$, $\gamma_1=1,\gamma_2=-1$. We thus have $\Gamma_m(r)=r-1>0$, as shown in the section on characteristic polynomials. We can thus find a $p(x)$ whose first two coefficients are $1$, $-1$ with no positive coefficients thereafter; this is the characteristic polynomial of a PLRR.
\end{proof}

Note that a ZLRR does not have a unique derived PLRR; the Zeroing Algorithm simply produces a PLRR whose characteristic polynomial takes the coefficients $1$, $-1$, a bunch of $0$'s, and up to $k$ nonzero terms at the end, where $k$ is the degree of the characteristic polynomial of the ZLRR. In fact, for any positive integer $n$ less than the principal root of a ZLRR, there exists a derived PLRR with leading coefficients $1,-n$; this is seen by taking $\gamma_2=-n$ in Corollary \ref{cor: conversion}. In Appendix \ref{app:list}, we provide an example conversion of a ZLRR to a PLRR, as well as a list of ZLRRs and their derived PLRR that comes from the Zeroing Algorithm.

%%%%%%%%%%%%%%%%%%%%%%%%%%%%%%%%%%%%%%%%%%%%%%%%%%%%%%%%%%%%%%%%%%%%%%%%%
%%%%%%%%%%%%%%%%%%%%%%%%%%%%%%%%%%%%%%%%%%%%%%%%%%%%%%%%%%%%%%%%%%%%%%%%%
%%%%%%%%%%%%%%%%%%%%%%%%%%%%%%%%%%%%%%%%%%%%%%%%%%%%%%%%%%%%%%%%%%%%%%%%%
\subsection{Fast Determination of Divergence Using the Zeroing Algorithm}

Finally, we have all of the tools necessary to prove our final result, which predicts the direction of divergence of a PLRS/ZLRS using its initial values. An example prediction is given in Appendix \ref{app:examples}.

\begin{proof}[Proof of Theorem \ref{thm:algorithm determination}]
We set $Q_0(x)=Q(x)$ and run the Zeroing Algorithm; we have proved that the sequence $q(1,t)$ follows the linear recurrence and has behavior determined by $Q_0(r)$. Thus, it suffices to show that $q(1,t)$ has the same initial values as $a_t$; explicitly, $q(1,t-1)=a_{t}$ for $1\le t\le k$.

We first notice, from the recurrences on $q(n,t)$ (equation \ref{eq: coeffrecur}), that
\begin{align}
    q(1,t)&\ = \ c_1\,q(1,t-1)+q(2,t-1)\nonumber\\
    &\ = \ c_1\,q(1,t-1)+c_2\,q(1,t-2)+q(3,t-2)\nonumber\\
    &\ \  \ \vdots\nonumber\\
    &\ = \ c_1\,q(1,t-1)+c_2\,q(1,t-2)+\cdots+c_t\,q(1,0)+q(t+1,0)\nonumber\\
    &\ = \ c_1\,q(1,t-1)+c_2\,q(1,t-2)+\cdots+c_t\,q(1,0)+(\alpha_{t+1}-d_{t+1}).
\end{align}

Now we proceed by strong induction. By construction, $q(1,0)=a_1$. For some $t$, assume $q(1,\tau-1)=a_{\tau}$ for all $1\le \tau <t$. We thus have
\begin{align}
    q(1,t)&\ = \ c_1\,q(1,t-1)+c_2\,q(1,t-2)+\cdots+c_t\,q(1,0)+(a_{t+1}-d_{t+1})\nonumber\\
    &\ = \ (c_1\,a_t+c_2\,a_{t-1}+\cdots+c_t\,a_1)+a_{t+1}-d_{t+1}\nonumber\\
    &\ = \ d_{t+1}+a_{t+1}-d_{t+1}\nonumber\\
    &\ = \ a_{t+1}
\end{align}
as desired.
\end{proof}

%%%%%%%%%%%%%%%%%%%%%%%%%%%%%%%%%%%%%%%%%%%%%%%%%%%%%%%%%%%%%%%%%%%%%%%%%%%%%%
%%%%%%%%%%%%%%%%%%%%%%%%%%%%%%%%%%%%%%%
%%%%%%%%%%%%%%%%%%%%%%%%%%%%%%%%%%%%%%%
%%%%%%%%%%%%%%%%%%%%%%%%%%%%%%%%%%%%%%%%%%%%%%%%%%%%%%%%%%%%%%%%%%%%%%%%%%%%%%
%%%%%%%%%%%%%%%%%%%%%%%%%%%%%%%%%%%%%%%%%5
\section{Investigating the Run-Time of the Zeroing Algorithm}\label{sec:runtime}

\par Consider the Unmodified Zeroing Algorithm as in the preceding section for an arbitrary ZLRR. Theorem \ref{thm:algorithmtermination} states that the eventual behavior of the Zeroing Algorithm can be determined solely by the sign of $Q_0(x)$. In this section we work to bound the run-time of the Zeroing Algorithm, and demonstrate that $Q_0(x)$ is also the primary determinant of the run-time. One difficulty with bounding the Zeroing Algorithm is that each coefficient of $Q_t(x)$ must diverge to negative infinity for the algorithm to terminate. Thus, each coefficient must be tracked. We begin by showing how all coefficients can be accounted for by keeping track of $q(1,t)$ alone.

\begin{proposition}\label{prop:q(1,t)}
After the step at which the sequence $q(1,t)$ becomes non-positive, the Zeroing Algorithm will terminate within the next $k-2$ steps.
\end{proposition}

\begin{proof}
\par We assume that $Q_0(r) < 0$, and the Zeroing Algorithm does indeed terminate. Consider the relationships between coefficient sequences starting with Equation \eqref{eq: coeffrecur}. Since we are dealing with ZLRRs, $c_1 = 0,$ we have $q(1,t) = q(2,t-1)$. Thus, it must be the case that $q(2,t)$ will become fully negative exactly one step before $q(1,t)$. Note also that $q(k,t)$ becomes fully negative one step after $q(1,t)$. Following this, in the worst case the coefficient sequences $q(k-1,t), \dots, q(3,t)$ will become fully negative in consecutive steps after $q(k,t)$. This amounts to a maximum of $k-2$ steps that the algorithm can take after $q(1,t)$ becomes fully negative.
\end{proof}

\begin{remark}\label{claim:q(1,t)Init}
 To determine the runtime, we see that it suffices to determine the iteration when $q(1,t)$ becomes nonpositive. Recall that $q(1,t)$ satisfies the recurrence specified by $P(x)$ (Lemma \ref{lem:satisfiesRecurrence}), and thus has a Binet expansion using the roots of $P(x)$ (Theorem \ref{thm:binetexpansion}), with the coefficients of those roots in the expansion being determined by the $k$ initial values $q(1,0),\dots,q(1,k-1)$. By equation \ref{eq: coeffrecur}, we can obtain the initial values
\begin{equation}\label{eq:initCond}
q(1,j) \ = \ \beta_{j+1} + \sum_{i = 2}^{j} c_i~q(1,j-i).
\end{equation}
for $0 \leq j \leq k-1$.
\end{remark}

\par Recall that $P(x)$ has a principal root $r$, which determines the behavior of $q(1,t)$, as it ``dominates'' the behavior of the other roots of $P(x)$ in the Binet expansion. Therefore, we now turn our attention to the \textit{principal coefficient}, which we define to be the coefficient of the principal root in the Binet expansion. For this principal coefficient determines the behavior of $r$ in the Binet expansion, and thus the behavior of $q(1,t)$. We begin with key notation. (Note that in the remainder of this section, we may refer to the principal root $r$ as $r_1$ for ease of indexing.)\\

\begin{definition}\label{def:esp}
We denote the $n$th degree \textbf{\textit{Elementary Symmetric Polynomial}} of $k$ items by
\begin{equation}\label{eq:esp}
S_n(x_1,\dots,x_k) \ = \ x_1 x_2 \cdots x_n + \dots + x_{k+1-n} \cdots x_k \ = \ \sum_{1 \leq i_1 < i_2 < \dots < i_n \leq k} x_{i_1} \cdots x_{i_n}
\end{equation}
where $1 \leq n \leq k$. If $n = 0$, we define $S_0(x_1,\dots,x_k) = 1.$ \\
\end{definition}

\begin{lemma}\label{lem:r1ToS1}
Consider $P(x)$ in \eqref{eq:P(x)}. Then we have
\begin{equation}\label{eq:r1ToS1}
r_1 \ = \ -S_1(r_2,\dots,r_k),
\end{equation}
and
\begin{equation}\label{eq:espRelation}
S_1(r_2,\dots,r_k) S_{n-1}(r_2,\dots,r_k) \ = \ S_n(r_2,\dots,r_k) + (-1)^n c_n.
\end{equation}
\end{lemma}

\begin{proof}
For an arbitrary polynomial of the form $a_k\, x^k + \cdots + a_1 \,x + a_0$ with roots $r_1,\dots,r_k$, Vieta's Formulas can be written as
\[S_n(r_1,\dots,r_k) \ = \ (-1)^n \,\frac{a_{k-n}}{a_k}, \quad \quad \mbox{for}~ 1\leq n \leq k.\]
Given the form of $P(x)$ this simplifies to
\begin{equation}\label{eq:vieta}
c_n \ = \ (-1)^{n+1} S_n(r_1,\dots,r_k)
\end{equation}
for $1 \leq n \leq k$.
Then we know that $c_1 = r_1 + r_2 + \cdots + r_k$. Thus \eqref{eq:r1ToS1} then follows from the fact that $c_1 = 0$, since we are dealing with the characteristic polynomial of a ZLRR.
\par Then we have
\begin{align}
S_n(r_1,\dots,r_k) &\ = \ \sum_{1 \leq i_1 < i_2 < \cdots < i_n \leq k} r_{i_1} \cdots r_{i_n} \nonumber\\ \nonumber\\
&\ = \ \sum_{2 \leq i_1 < i_2 < \cdots < i_n \leq k} r_{i_1} \cdots r_{i_n} + \sum_{2 \leq i_1 < i_2 < \cdots < i_{n-1} \leq k} r_{1} r_{i_1} \cdots r_{i_{k-1}}  \nonumber\\ \nonumber\\
&\ = \ S_n(r_2,\dots,r_k) + r_1 S_{n-1}(r_2,\dots,r_k) \nonumber\\
&\ = \ S_n(r_2,\dots,r_k) - S_1(r_2,\dots,r_k) S_{n-1}(r_2,\dots,r_k), \label{eq:cmRelation}
\end{align}
which implies,

\begin{align}
S_1(r_2,\dots,r_k) S_{n-1}(r_2,\dots,r_k) &\ = \ S_n(r_2,\dots,r_k) - S_n(r_1,\dots,r_k) \nonumber \\
&\ = \ S_n(r_2,\dots,r_k) + (-1)^n (-1)^{n+1} S_n(r_1,\dots,r_k) \nonumber \\
&\ = \ S_n(r_2,\dots,r_k) + (-1)^n \cdot c_n.
\end{align}

\end{proof}

%%%%%%%%%%%%%%%%%%%%%%%%%%%%%%%%%%%%%%%%%%%%%
\begin{thm}\label{thm:principalCoefficient}
Consider $P(x)$ in \eqref{eq:P(x)}. Suppose the roots of $P(x)$, $r_1,\dots,r_k$, each have multiplicity 1, and without loss of generality, suppose $r_1 > |r_2| > \cdots > |r_k|$, with $r_1$ being the principal root. Then, considering the Binet expansion $q(1,t) = a_1\, r_1^t + \cdots + a_k\, r_k^t$
we have
\begin{equation}\label{eq:a1Theorem}
a_1 = \frac{Q_0(r_1)}{\prod_{i\,=\,2}^{k}(r_1 - r_i)}.
\end{equation}
\end{thm}

\begin{proof}
Using Equation \eqref{eq:initCond} to find the initial $k$ values of $q(1,t)$, we note that the coefficients $a_1,\cdots,a_k$ are the solutions to the following linear system:
\begin{equation}
\begin{pmatrix}
1 & 1 & 1 & \cdots & 1 \\
r_1 & r_2 & r_3 & \cdots & r_k \\
r_1^2 & r_2^2 & r_3^2 & \cdots & r_k^2 \\
\vdots & \vdots & \vdots & \ddots & \vdots\\
r_1^{k-1} & r_2^{k-1} & r_3^{k-1} & \cdots & r_k^{k-1}
\end{pmatrix}
\begin{pmatrix}
a_1 \\
a_2 \\
a_3 \\
\vdots \\
a_k
\end{pmatrix}
=
\begin{pmatrix}
\beta_1 \\
\beta_2 \\
\beta_2 + c_2 \beta_1 \\
\vdots\\
\beta_{k} + \sum_{i = 2}^{k-1} c_i ~ q(1,k-1-i)
\end{pmatrix}.
\end{equation}
\par To find $a_1$, we can use Cramer's Rule. Let $A$ denote the matrix of roots. If we let $A_1$ be the matrix formed by substituting the first column of $A$ with the column vector of initial terms of $q(1,t)$, then we have $a_1 = \det(A_1)/\det(A)$. Because $A$ is a Vandermonde matrix we have
\begin{equation}
\det(A) = \prod_{1\leq i < j \leq k} (r_j - r_i).
\end{equation}
We can then use the Laplace expansion to obtain
\begin{equation}\label{eq:A1}
   \det(A_1) = \sum_{n\,=\,1}^k ~ q(1,n-1)(-1)^{n+1} \det(M_{n1}),
\end{equation}
where $M_{n1}$ is the $n, 1$ minor of $A_1$.

\par Note that each minor -- except for $M_{k1}$ -- in \eqref{eq:A1} is not a Vandermonde matrix due to its missing row of geometric terms. However, the determinants of these ``punctured'' Vandermonde matrices have a similar form to the determinant of a regular Vandermonde matrix involving elementary symmetric polynomials of the roots. By the results found in \cite{KKL}, we can write
\begin{equation}
\det(M_{n1}) \ = \ S_{k-n}(r_2,\dots,r_k) \cdot \prod_{2\leq i < j \leq k} (r_j - r_i).
\end{equation}
So, we have
\begin{align}
 \det(A_1) &\ = \ \sum_{n=1}^k ~ \left[q(1,n-1) \cdot (-1)^{n+1} \cdot S_{k-n}(r_2,\dots,r_k) \cdot \prod_{2\leq i < j \leq k} (r_j - r_i)\right] \nonumber \\
 &\ = \ \left[\prod_{2\leq i < j \leq k} (r_j - r_i)\right] \cdot \sum_{n\,=\,1}^k ~ \left[q(1,n-1) \cdot (-1)^{n+1} \cdot S_{k-n}(r_2,\dots,r_k)\right].
 \end{align}
 So, we can solve for $a_1$.
 \begin{align}
 a_1 \ = \ \frac{\det(A_1)}{\det(A)} &\ = \ \frac{\sum_{n\,=\,1}^k ~ q(1,n-1) \cdot (-1)^{n+1} \cdot S_{k-n}(r_2,\dots,r_k)}{\prod_{i\,=\,2}^{k}(r_i - r_1)} \nonumber\\ \nonumber \\
&\ = \ \frac{\sum_{n=1}^k ~ q(1,n-1) \cdot (-1)^{n+1} \cdot S_{k-n}(r_2,\dots,r_k)}{(-1)^{k+1} \cdot \prod_{i=2}^{k}(r_1 - r_i)}. \label{eq:intermediate}
\end{align}
Next, we shall demonstrate that
\begin{equation}\label{eq:induction}
(-1)^{m+1} \cdot \sum_{n\,=\,1}^m \beta_n\, r_1^{m-n} \ = \  \sum_{n\,=\,1}^m ~ q(1,n-1) \cdot (-1)^{n+1} \cdot S_{m-n}(r_2,\dots,r_k)
\end{equation}
by induction on $m$. (However, we note that implicitly $m \leq k$, since we have not defined $\beta_i$ where $i > k$.)\\

When $m=k$, this implies that
\begin{equation}
(-1)^{k+1} \cdot Q_0(r_1) = \sum_{n\,=\,1}^{k} ~ \left[q(1,n-1) \cdot (-1)^{n+1} \cdot S_{m-n}(r_2,\dots,r_k)\right].
\end{equation}

\par Therefore, by simplifying \eqref{eq:intermediate} we have \eqref{eq:a1Theorem}. \\

\textbf{\textit{Base Case ($m=3$):}} We have, by plugging in \eqref{eq:vieta} and \eqref{eq:cmRelation},
\begin{align}
&\sum_{n\,=\,1}^3 ~ q(1,n-1) \cdot (-1)^{n+1} \cdot S_{3-n}(r_2,\dots,r_k) \nonumber\\
&\ \ \ \ \ \ = \ \beta_1 \,S_{2}(r_2,\dots,r_k) - \beta_2\, S_{1}(r_2,\dots,r_k) + \beta_1\, c_2 + \beta_3\nonumber\\
&\ \ \ \ \ \ = \ \beta_1\, S_{2}(r_2,\dots,r_k) - \beta_2\, S_{1}(r_2,\dots,r_k) + \beta_1\nonumber\\
&\ \ \ \ \ \ = \ \beta_1 \,r_1^2 + \beta_2 \,r_1 + \beta_3\nonumber\\
&\ \ \ \ \ \ = \  (-1)^{3+1} \cdot \sum_{n\,=\,1}^3 \beta_n\, r_1^{3-n}
\end{align}

\par \textbf{\textit{Inductive Step:}} Assume \eqref{eq:induction} holds for all $m' < m$. Then for $m' = m-1$ we have

\begin{equation}
(-1)^{m} \cdot \sum_{n=1}^{m-1} \beta_n r_1^{m-1-n} \ = \  \sum_{n=1}^{m-1} ~ q(1,n-1) \cdot (-1)^{n+1} \cdot S_{m-1-n}(r_2,\dots,r_k).
\end{equation}
So, we have
\begin{align}
&(-1)^{m+1} \cdot \sum_{n=1}^{m} \beta_n r_1^{m-n} \ = \ -r_1 \cdot \sum_{n=1}^{m-1} ~ \left[q(1,n-1) \cdot (-1)^{n+1} \cdot S_{m-1-n}(r_2,\dots,r_k)\right] + \beta_m \nonumber\\
&\ = \ S_1(r_2,\dots,r_k) \cdot \sum_{n=1}^{m-1} ~ \left[q(1,n-1) \cdot (-1)^{n+1} \cdot S_{m-1-n}(r_2,\dots,r_k)\right] + \beta_m \nonumber\\
&\ = \ \sum_{n=1}^{m-1} ~ \left[q(1,n-1) \cdot (-1)^{n+1} \cdot \left[S_{m-n}(r_2,\dots,r_k) + (-1)^{m-n} \cdot c_{m-n}\right]\right] + \beta_m \nonumber\\
&\ = \ \sum_{n=1}^{m-1} ~ \left[q(1,n-1) \cdot (-1)^{n+1} \cdot S_{m-n}(r_2,\dots,r_k)\right] + (-1)^{m+1} \cdot \sum_{n=1}^{m-1} ~ \left[q(1,n-1) \cdot c_{m-n}\right]+ \beta_m \nonumber\\
&\ = \ \sum_{n=1}^{m} ~ \left[q(1,n-1) \cdot (-1)^{n+1} \cdot S_{m-n}(r_2,\dots,r_k)\right] \quad \quad \mbox{by \eqref{eq:initCond}.}
\end{align}

\par 
\end{proof}

\begin{corollary}\label{cor:rootsCloseTo1}
As $r \to 1$, we have $t_0 \to \infty$, where $t_0$ denotes the number of steps taken after the modified Zeroing Algorithm of Lemma \ref{lem:modifiedAlgorithm} reverts back to the unmodified form.
\end{corollary}
\begin{remark}
Corollary \ref{cor:rootsCloseTo1} tells us that ZLRRs with principal roots closest to $1$ will take the longest to convert into a derived PLRR.
\end{remark}
\begin{proof}
    From Lemma \ref{lem:modifiedAlgorithm} we know that  $Q_m(r) = -r^k \cdot \Gamma_m(r)$. This is the iteration when the modified Zeroing Algorithm reverts to the unmodified Zeroing Algorithm. Thus for $m = 2, \gamma_1 = 1,~\mbox{and}~ \gamma_2 = -1,$ we have $Q_2(r) = -r^k~ (r-1).$ Recall that this configuration of the modified Zeroing Algorithm results in the "minimal" derived PLRR of a given ZLRR. So, as $r\to 1$, we have $Q_2(r)\to 0^-$. (Recall that $Q_2(r)$ is equivalent to $Q_0(r)$ of the unmodified Zeroing Algorithm.) Then by Theorem \ref{thm:principalCoefficient} we know that since $Q_0(r) \to 0^-$, it also follows that $a_1 \rightarrow 0$ and thus the principal root of the Binet expansion of $q(1,t)$ takes longer and longer to dominate, implying that $t_0 \rightarrow \infty$.
    \end{proof}

Note that the above conclusions only apply to ZLRRs whose roots have multiplicity $1$. Extending Theorem \ref{thm:principalCoefficient} to cover ZLRRs with roots of any multiplicity is more difficult because the Binet expansion of $q(1,t)$ becomes more complicated, which negates the use of Vandermonde matrices in the proof of Theorem \ref{thm:principalCoefficient}. In the more general case, we conjecture the following.

\begin{conjecture}\label{conj:fullBinetConj}
If the roots of $P(x)$ are $r_1,r_2,\dots,r_i$, with respective multiplicities $1$, $m_2$, $\dots$, $m_i$ such that $m_j \geq 1$ with $2 \leq j \leq i \leq k$, then for the coefficient $a_1$ of the principal root in the Binet expansion of $q(1,t)$ we have
\begin{equation}
a_1 = \frac{Q_0(r_1)}{\prod_{j\,=\,2}^{i}(r_1 - r_j)^{m_j}}.
\end{equation}
\end{conjecture}

In order to work towards finding the true bound of the Zeroing Algorithm, we also wish to quantify the relationship between $Q_0(r)$ and the run-time beyond the general tendencies that our current results provide. Notably, Theorem \ref{thm:principalCoefficient} suggests that as $Q_0(r) \to 0^-$, the run-time becomes unbounded, since the principal root in the Binet expansion of $q(1,t)$ will take longer and longer to dominate.

Some experimentation provides a way to visualize the relationship; see Figure \ref{sim}.

\begin{figure}[ht!]
	\begin{center}
	\includegraphics[scale=.4]{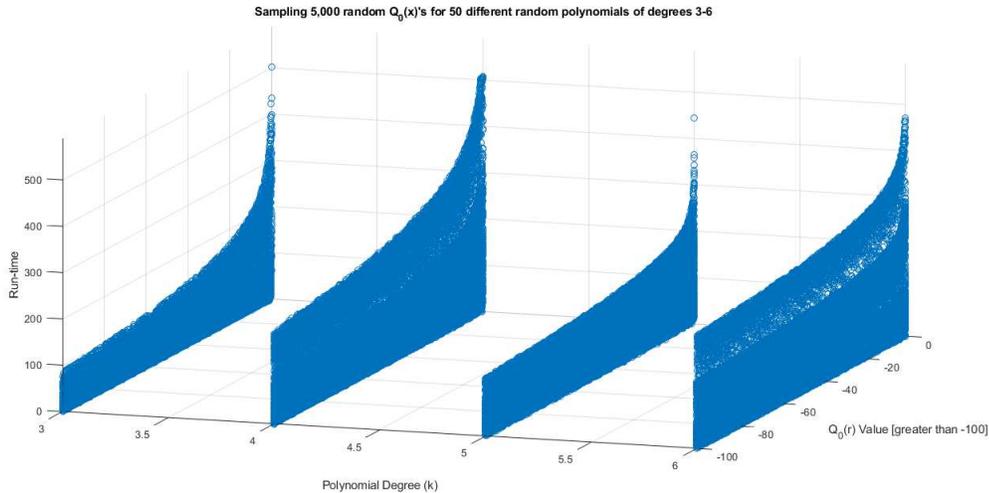}
	\end{center}
    \caption{The results of a MATLAB simulation that generated 50 random $P(x)$ polynomials for each degree 3 to 6, and sampled 5,000 random $Q_0(x)$'s for each random $P(x)$. A strong inverse relationship can be seen between $Q_0(r)$ and the run-time.}
    \label{sim}
\end{figure}

\par The above observations inspire us to conjecture the following concerning the bound of the Zeroing Algorithm:
\begin{conjecture}\label{conj:trueBound}
$Q_0(r)$ and the run-time have an inverse relationship.
\end{conjecture}

%%%%%%%%%%%%%%%%%%%%%%%%%%%%%%%%%%%%%%%%%%%%%%%%%%%%%%%%%%%%%%%%%%%%%%%%%%%%%%%%%%%%%%%%%%%%%%%%%%%%%%%%%%%%%%%%%%%%%%%%%%%%%%%%%%%%%%%%%%%%%%%%%%%%%%
%%%%%%%%%%%%%%%%%%%%%%%%%%%%%%%%%%%%%%%%%%%%%%%%%%%%%%%%%%%%%%%%%%%%%%%%%%%%%%%%%%%%%%%%%%%%%%%%%%%%%%%%%%%%%%%%%%%%%%%%%%%%%%%%%%%%%%%%%%%%%%%%%%%%%%
%%%%%%%%%%%%%%%%%%%%%%%%%%%%%%%%%%%%%%%%%%%%%%%%%%%%%%%%%%%%%%%%%%%%%%%%%%%%%%%%%%%%%%%%%%%%%%%%%%%%%%%%%%%%%%%%%%%%%%%%%%%%%%%%%%%%%%%%%%%%%%%%%%%%%%
\section{Conclusion and Future work}\label{sec:conclusion}

We have introduced two distinct ways to consider decompositions arising from ZLRSes.

\begin{itemize}
    \item As we saw from the first method, we can define decompositions in such a way that we have existence, but not uniqueness. Is there a different definition such that we have uniqueness, but not existence? Is it possible to have both existence and uniqueness, or can we prove that having both is generally impossible for ZLRSes? \\ \

    \item In terms of bounding the run-time of the Zeroing Algorithm, the next steps are to prove Conjectures \ref{conj:fullBinetConj} and \ref{conj:trueBound}, or similar run-time results if it turns out that these do not hold. \\ \

    \item The Zeroing Algorithm has proven a powerful tool for studying linear recurrences analytically; how does it provide information on more discrete questions such as decompositions with ZLRSes? Are specific sets of initial values necessary for a decomposition to have desirable properties? Are there such properties that are inherent in the recurrence relation itself, rather than being contingent on a specific sequence produced by the initial values?
\end{itemize}

%%%%%%%%%%%%%%%%%%%%%%%%%%%%%%%%%%%%%%%%%%%%%%%%%%%%%%%%%%%%%%%%%%%%%%%%%%%%%%%%%%%%%%%%%%%%%%%%%%%%%%%%%%%%%%%%%%%%%%%%%%%%%%%%%%%%%%%%%%%%%%%%%%%%%%%%
%%%%%%%%%%%%%%%%%%%%%%%%%%%%%%%%%%%%%%%%%%%%%%%%%%%%%%%%%%%%%%%%%%%%%%%%%%%%%%%%%%%%%%%%%%%%%%%%%%%%%%%%%%%%%%%%%%%%%%%%%%%%%%%%%%%%%%%%%%%%%%%%%%%%%%%%
%%%%%%%%%%%%%%%%%%%%%%%%%%%%%%%%%%%%%%%%%%%%%%%%%%%%%%%%%%%%%%%%%%%%%%%%%%%%%%%%%%%%%%%%%%%%%%%%%%%%%%%%%%%%%%%%%%%%%%%%%%%%%%%%%%%%%%%%%%%%%%%%%%%%%%%%
\appendix

%%%%%%%%%%%%%%%%%%%%%%%%%%%%%%%%%%%%%%%%%%%%%%%%%%%%%%%%%%%%%%%%%%%%%%%%%%%%%%%%%%%%%%%%%%%%%%%%%%%%%%%%%%%%%%%%%%%%%%%%%%%%%%%%%%%%%%%%%%%%%%%%%%%%%%%%
%%%%%%%%%%%%%%%%%%%%%%%%%%%%%%%%%%%%%%%%%%%%%%%%%%%%%%%%%%%%%%%%%%%%%%%%%%%%%%%%%%%%%%%%%%%%%%%%%%%%%%%%%%%%%%%%%%%%%%%%%%%%%%%%%%%%%%%%%%%%%%%%%%%%%%%%
%%%%%%%%%%%%%%%%%%%%%%%%%%%%%%%%%%%%%%%%%%%%%%%%%%%%%%%%%%%%%%%%%%%%%%%%%%%%%%%%%%%%%%%%%%%%%%%%%%%%%%%%%%%%%%%%%%%%%%%%%%%%%%%%%%%%%%%%%%%%%%%%%%%%%%%%
\section{Some Examples of Running the Zeroing Algorithm}\label{app:examples}

Consider the recurrence relation
\[H_{n+1}=2H_{n-1}+H_{n-2},\]
which has characteristic polynomial $P(x)=x^3-2x-1$ (principal root $r=(1+\sqrt{5})/2$), where we have the coefficients $c_1=0,c_2=2,c_3=1$. Suppose we are given $\beta_1=3$, $\beta_2=-2$, $\beta_3=-5$; we run the algorithm as follows:

\begin{tabular}{ r r r r r r r r r r r}

$3$ & $-2$ & $-5$ &   &  &&&& $Q_0(x)=3x^2-2x-5$ \\
 $-3$ & $0$ & $6$ & $3$ &    &    &&&\\
  \cline{1-4}
  & $-2$ & $1$ & $3$ &    &    & &&$Q_1(x)=-2x^2\,+\,x\,+\,3$\\
  &  $2$ & $0$  & $-4$ & $-2$ &  &&&\\
  \cline{2-5}
  &   &  $1$ & $-1$ & $-2$ &   &&&$Q_2(x)=x^2\,-\,x-2$\\
  &   & $-1$  & $0$ & $2$ & $1$ &&&\\
  \cline{3-6}
  &   &  & $-1$ & $0$ & $1$  &&&$Q_3(x)=-x^2\,-\,0x+1$\\
  &   &  & $1$ & $0$ & $-2$ & $-1$ &&\\
  \cline{4-7}
  &   &  & & $0$ & $-1$  & $-1$ &&$Q_4(x)=0x^2\,-\,x-1$\\
\end{tabular}\\ \

We reach termination on step $4$, since $Q_4$ does not have positive coefficients. Note that the Zeroing Algorithm is named for the first (omitted) coefficient of $0$ following each step. \\

Suppose that given the same recurrence relation, and initial values $a_0=3,a_1=-2,a_3=1$, we wish to determine whether the recurrence sequence diverges to negative infinity.\\

Using the method introduced in Theorem \ref{thm:algorithm determination}, we first determine the values of

\[d_2\ = \ a_1c_1\ = \ 0,\ \ \ \  \ d_3\ =\ a_1c_2 + a_2c_1\ = \ 6,\]
from which we construct
\[Q(x)\ = \ a_1x^2+(a_2-d_2)x+(a_3-d_3)=3x^2-2x-5.\]
We have $Q(r)=3r^2-2r-5=3(r+1)-2r-5=r-2<0$, which predicts that $\{a_n\}$ diverges to negative infinity.\\

Manually computing the terms gives
\[3,\ -2,\ 1,\ -1,\ 0,\ -1,\ -1,\ -2,\ -3,\ -5,\ -8,\ -13,\ \dots,\]
which confirms our prediction.

%%%%%%%%%%%%%%%%%%%%%%%%%%%%%%%%%%%%%%%%%%%%%%%%%%%%%%%%%%%%%%%%%%%%%%%%%%%%%%%%%%%%%%%%%%%%%%%%%%%%%%%%%%%%%%%%%%%%%%%%%%%%%%%%%%%%%%%%%%%%%%%%%%%%%%%%
%%%%%%%%%%%%%%%%%%%%%%%%%%%%%%%%%%%%%%%%%%%%%%%%%%%%%%%%%%%%%%%%%%%%%%%%%%%%%%%%%%%%%%%%%%%%%%%%%%%%%%%%%%%%%%%%%%%%%%%%%%%%%%%%%%%%%%%%%%%%%%%%%%%%%%%%
%%%%%%%%%%%%%%%%%%%%%%%%%%%%%%%%%%%%%%%%%%%%%%%%%%%%%%%%%%%%%%%%%%%%%%%%%%%%%%%%%%%%%%%%%%%%%%%%%%%%%%%%%%%%%%%%%%%%%%%%%%%%%%%%%%%%%%%%%%%%%%%%%%%%%%%%
\newpage
\section{List of ZLRRs and derived ZLRRs}\label{app:list}

\noindent \textbf{1.} Recurrence: $G_{n+1}=G_{n-1}\,+\,G_{n-2}, \ P(x)=x^3\,-\,0\,x^2\,-\,x\,-\,1.$\\

\begin{tabular}{ r r r r r r r r r r r}

$\gamma_1=1$ & 0 & -1 & -1  &  &&&& $Q_1(x)=0x^2-x-1$ \\
  & -1&  0 & 1 & 1   &    &&&\\
  \cline{2-5}
  & $\gamma_2=-1$& -1 & 0 & 1   &    & &&$Q_2(x)=-x^2\,+\,0x\,+\,1$\\
  &   & 1  & 0 & -1& -1  &&&\\
  \cline{3-6}
  &   &  $\gamma_3=0$ & 0 & 0& -1  &&&$Q_3(x)=0x^2\,+\,0x-1$\\
\end{tabular}\\ \

\noindent Derived characteristic polynomial: $x^5\,-\,x^4\,-\,0\,x^3\,-\,0\,x^2\,-\,0\,x\,-\,1$, which corresponds to the derived PLRR $H_{n+1}=H_n\,+\,H_{n-4}$.\\

\noindent \textbf{2.} Current ZLRR: $G_{n+1} = G_{n-1}\,+\,G_{n-2}\,+\,G_{n-3}$.\\

\noindent Current characteristic polynomial: $x^4\,-\,x^2\,-\,x\,-\,1$.\\

\noindent Derived characteristic polynomial: $x^6\,-\,x^5\,-\,x^2\,-\,1$.\\

\noindent Derived PLRR: $H_{n+1}=H_n\,+\, H_{n-3}\,+\,H_{n-5}$.\\

\noindent \textbf{3.} Current ZLRR: $G_{n+1} = 2\,G_{n-1}\,+\,2\,G_{n-2}$.

\noindent Current characteristic polynomial: $x^3\,-\,2\,x\,-\,2$.\\

\noindent Derived characteristic polynomial: $x^5\,-\,x^4\,-\,2\,x\,-\,4$.\\

\noindent Derived PLRR: $H_{n+1} = H_n\,+\, 2\,H_{n-3}\,+\,4\,H_{n-4}$.\\

\noindent \textbf{4.} Current ZLRR: $G_{n+1} = 19G_{n-1} \,+\, 38G_{n-4}$.\\

\noindent Current characteristic polynomial: $x^5\,-\,19x^3\,-\,38$.\\

\noindent  Derived characteristic polynomial: $x^{29}\,-\,x^{28}\,-\,310601172680577\,x^4 \,-\,40586681545596725\,x^3\\ \,-\,4277914985538462\,x^2 \,-\,170201741455942\,x \,-\,81203021913963806$.\\

\noindent Derived PLRR: $H_{n+1} = H_n\,+\,310601172680577\,H_{n-24} \,+\,40586681545596725\,H_{n-25}\\ \,+\,4277914985538462\,H_{n-26} \,+\,170201741455942\,H_{n-27} \,+\,81203021913963806\,H_{n-28}$.\\

\noindent \textbf{5.} Current ZLRR: $G_{n+1} = 6\,G_{n-1} \,+\, 3\,G_{n-2} \,+\, 5\,G_{n-3}$.\\

\noindent Current characteristic polynomial: $x^4\,-\,6\,x^2\,-\,3\,x\,-\,5$.\\

\noindent Derived characteristic polynomial: $x^{10}\,-\,x^9\,-\,69\,x^3\,-\,1669\,x^2\,-\,722\,x\,-\,1245$.\\

\noindent Derived PLRR: $H_{n+1} = H_n\,+\, 69\,H_{n-6}\,+\,1669\,H_{n-7}\,+\,722\,H_{n-8}\,+\,1245\,H_{n-9}$.\\

\noindent \textbf{6.} Current ZLRR: $G_{n+1} = G_{n-2} \,+\, G_{n-3}$.\\

\noindent Current characteristic polynomial: $x^4\,-\,x\,-\,1$.\\

\noindent Derived characteristic polynomial: $x^{20}\,-\,x^{19}\,-\,4\,x^3\,-\,x^2\,-\,1$.\\

\noindent Derived PLRR: $H_{n+1} = H_n\,+\, 4\,H_{n-16}\,+\,H_{n-17}\,+\,H_{n-19}$.\\

\noindent \textbf{7.} Current ZLRR: $G_{n+1} = 3\,G_{n-2} \,+\, G_{n-3} \,+\, 3\,G_{n-4}$.\\

\noindent Current characteristic polynomial: $x^5\,-\,3\,x^2\,-\,x\,-\,3$.\\

\noindent Derived characteristic polynomial: $x^{13}\,-\,x^{12}\,-\,14\,x^4\,-\,3\,x^3\,-\,54\,x^2\,-\,4\,x\,-\,39$.\\

\noindent Derived PLRR: $H_{n+1} = H_n\,+\, 14\,H_{n-8}\,+\,3\,H_{n-9}\,+\,54\,H_{n-10}\,+\,4\,H_{n-11}\,+\,39\,H_{n-12}$.\\

\noindent \textbf{8.} Current ZLRR: $G_{n+1} = G_{n-2} \,+\, G_{n-19}$.\\

\noindent Current characteristic polynomial: $x^{20}\,-\,x^{17}\,-\,1$.\\

\noindent Derived characteristic polynomial: \small$x^{358}\,-\,x^{357}\,-\,4000705295\,x^{19} \,-\,7080648306\,x^{18} \,-\,575930712\,x^{17} \,-\,1937068817\,x^{16} \,-\,1082811308\,x^{15} \,-\,92014103\,x^{14} \,-\,2546102784\,x^{13} \,-\,1062101754\,x^{12} \,-\,372938426\,x^{11} \,-\,3264026504\,x^{10} \,-\,996542899\,x^9 \,-\,834914708\,x^8 \,-\,4089249024\,x^7 \,-\,890353375\,x^6 \,-\,1541366894\,x^5 \,-\,5013188421\,x^4 \,-\,759208181x^3\,-\,2567648478\,x^2 \,-\,6018966637\,x\,-\,635668820$.\\ \normalsize

\noindent Derived PLRR: \small$H_{n+1} = H_n\,+\, 4000705295\,H_{n-338} \,+\,7080648306\,H_{n-339} \,+\,575930712\,H_{n-340} \,+\,1937068817\,H_{n-341} \,+\,1082811308\,H_{n-342} \,+\,92014103\,H_{n-343} \,+\,2546102784\,H_{n-344} \,+\,1062101754\,H_{n-345} \,+\,372938426\,H_{n-346} \,+\,3264026504\,H_{n-347} \,+\,996542899\,H_{n-348} \,+\,834914708\,H_{n-349} \,+\,4089249024\,H_{n-350} \,+\,890353375\,H_{n-351} \,+\,1541366894\,H_{n-352} \,+\,5013188421\,H_{n-353}\,+\,759208181\,H_{n-354} \,+\,2567648478\,H_{n-355} \,+\,6018966637\,H_{n-356} \,+\,635668820\,H_{n-357} $.\\ \normalsize

\noindent \textbf{9.} Current ZLRR: $G_{n+1} = G_{n-2} \,+\, G_{n-19}\,+\,G_{n-20}$.\\

\noindent Current characteristic polynomial: $x^{21}\,-\,x^{18}\,-\,x\,-\,1$.\\

\noindent Derived characteristic polynomial: $x^{156}\,-\,x^{155}\,-\,16626\,x^{20} \,-\,6\,x^{19} \,-\,16814\,x^{18} \,-\,4094\,x^{17} \,-\,1037\,x^{16} \,-\,6777\,x^{15} \,-\,5088\,x^{14} \,-\,1849\,x^{13} \,-\,9106\,x^{12} \,-\,6334\,x^{11} \,-\,3060\,x^{10} \,-\,12166\,x^9 \,-\,7932\,x^8 \,-\,4851\,x^7 \,-\,16190\,x^6 \,-\,10031\,x^5 \,-\,7482\,x^4 \,-\,21483\,x^3 \,-\,12839\,x^2 \,-\,11312\,x \,-\,11809$.\\

\noindent Derived PLRR: $H_{n+1} = H_n\,+\, 16626\,H_{n-135} \,+\,6\,H_{n-136}\,+\,16814\,H_{n-137} \,+\,4094\,H_{n-138}\,+\,1037\,H_{n-139} \,+\,6777\,H_{n-140} \,+\,5088\,H_{n-141} \,+\,1849\,H_{n-142} \,+\,9106\,H_{n-143} \,+\,6334\,H_{n-144} \,+\,3060\,H_{n-145} \,+\,12166\,H_{n-146} \,+\,7932\,H_{n-147} \,+\,4851\,H_{n-148} \,+\,16190\,H_{n-149} \,+\,10031\,H_{n-150} \,+\,7482\,H_{n-151} \,+\,21483\,H_{n-152} \,+\,12839\,H_{n-153} \,+\,11312\,H_{n-154} \,+\,11809\,H_{n-155}$.\\

\noindent \textbf{10.} Current ZLRR: $G_{n+1} = G_{n-1} \,+\, 2\,G_{n-2} \,+\, 2\,G_{n-4} \,+\, 3\,G_{n-5}$.\\

\noindent Current characteristic polynomial: $x^6\,-\,x^4\,-\,2\,x^3\,-\,2\,x\,-\,3$.\\

\noindent Derived characteristic polynomial: $x^{11}\,-\,x^{10}\,-\,2\,x^5\,-\,2\,x^4\,-\,15\,x^3\,-\,x^2\,-\,7\,x\,-\,15$.\\

\noindent Derived PLRR: $H_{n+1} = H_n\,+\, 2\,H_{n-5}\,+\,2\,H_{n-6}\,+\,15\,H_{n-7}\,+\,H_{n-8}\,+\,7\,H_{n-9}\,+\,15\,H_{n-10}$.\\

\noindent \textbf{11.} Current ZLRR: $G_{n+1} = 40\,G_{n-3} \,+\, 52\,G_{n-4}$.\\

\noindent Current characteristic polynomial: $x^5\,-\,40\,x\,-\,52$.\\

\noindent Derived characteristic polynomial: $x^{25}\,-\,x^{24}\,-\,555888384\,x^4 \,-\,1064960000\,x^3 \,-\,519168000\,x^2 \,-\,3308595200\,x \,-\,4535145472$.\\

\noindent Derived PLRR: $H_{n+1} = H_n\,+\, 555888384\,H_{n-20} \,+\,1064960000\,H_{n-21} \,+\,519168000\,H_{n-22} \\
\,+\,3308595200\,H_{n-23} \,+\,4535145472\,H_{n-24}$.\\

\noindent \textbf{12.} Current ZLRR: $G_{n+1} = G_{n-8} \,+\, G_{n-9}$.\\

\noindent Current characteristic polynomial: $x^{10}\,-\,x\,-\,1$.\\

\noindent Derived characteristic polynomial: \small$x^{488}\,-\,x^{487}\,-\,7634770044678\,x^9 \,-\,16848326467063\,x^8 \,-\,\\ 25319805215106\,x^7 \,-\,29495744687667\,x^6 \,-\,27304765351108\,x^5 \,-\,19325535741204\,x^4 \,-\,8910253837548\,x^3 \,-\,1049595609091\,x^2 \,-\,321640563521\,x \,-\,1106933774826$.\\ \normalsize

\noindent Derived PLRR: \small $H_{n+1} = H_n\,+\,7634770044678\,H_{n-478} \,+\,16848326467063\,H_{n-479} \,+\,\\ 25319805215106\,H_{n-480} \,+\,29495744687667\,H_{n-481}\,+\,27304765351108\,H_{n-482} \,+\,19325535741204\,H_{n-483} \,+\,8910253837548\,H_{n-484} \,+\,1049595609091\,H_{n-485} \,+\,321640563521\,H_{n-486} \,+\,1106933774826\,H_{n-487}$.\\ \normalsize

\noindent \textbf{13.} Current ZLRR: $G_{n+1}=G_{n-2}\,+\,G_{n-4}\,+\,G_{n-6}$.\\

\noindent Current characteristic polynomial: $x^7\,-\,x^4\,-\,x^2\,-\,1$.\\

\noindent Derived characteristic polynomial: $x^{23}\,-\,x^{22}\,-\,x^6 \,-\,6\,x^5 \,-\,x^4 \,-\,6\,x^3 \,-\,x^2 \,-\,3\,x \,-\,2$.\\

\noindent Derived PLRR: $H_{n+1} = H_n\,+\, H_{n-16}\,+\,6\,H_{n-17}\,+\,H_{n-18}\,+\,6\,H_{n-19}\,+\,H_{n-20}\,+\,3\,H_{n-21}\,+\,2\,H_{n-22}$.\\

\noindent \textbf{14.} Current ZLRR: $G_{n+1}=3\,G_{n-1}\,+\,5\,G_{n-2}$.\\

\noindent Current characteristic polynomial: $x^3\,-\,3\,x\,-\,5$.\\

\noindent Derived characteristic polynomial: $x^5\,-\,x^4\,-\,2\,x^2\,-\,4\,x\,-\,15$.\\

\noindent Derived PLRR: $H_{n+1} = H_n\,+\, 2\,H_{n-2}\,+\,H_{n-3}\,+\,15\,H_{n-4}$.\\

\noindent \textbf{15.} Current ZLRR: $G_{n+1}=G_{n-6}\,+\,G_{n-12}$.\\

\noindent Current characteristic polynomial: $x^{13}\,-\,x^6\,-\,1$.\\

\noindent Derived characteristic polynomial: $x^{572}\,-\,x^{571}\,-\,141734291356872\,x^{12} \,-\,1386240086076478\,x^{11} \,-\,3383864145243271\,x^{10} \,-\,4628373080436668\,x^9 \,-\,4069191511013055\,x^8 \,-\,2094637579574813\,x^7 \,-\,395154232336030\,x^6 \,-\,528518791146011\,x^5 \,-\,1761055564629423\,x^4 \,-\,2792877805797871\,x^3\\ \,-\,2780671348399214\,x^2 \,-\,1681201891412681\,x \,-\,401879825813162$.\\

\noindent Derived PLRR: $H_{n+1} = H_n\,+\, 141734291356872\,H_{n-559} \,+\,1386240086076478\,H_{n-560}\\ \,+\,3383864145243271\,H_{n-561} \,+\,4628373080436668\,H_{n-562} \,+\,4069191511013055\,H_{n-563}\\ \,+\,2094637579574813\,H_{n-564} \,+\,395154232336030\,H_{n-565} \,+\,528518791146011\,H_{n-566}\\ \,+\,1761055564629423\,H_{n-567} \,+\,2792877805797871\,H_{n-568} \,+\,2780671348399214\,H_{n-569}\\ \,+\,1681201891412681\,H_{n-570} \,+\,401879825813162\,H_{n-571} $.\\

\noindent \textbf{16.} Current ZLRR: $G_{n+1}=G_{n-9}\,+\,G_{n-10}$.\\

\noindent Current characteristic polynomial: $x^{11}-x-1$.\\

\noindent Derived characteristic polynomial: \small $x^{665}\,-\,x^{664}\,-\,17581679276200473\,x^{10} \,-\,43065699679149511\,x^9 \,-\,70765959937154578\,x^8 \,-\,91624450164084254\,x^7 \,-\,98016133194347743\,x^6 \,-\,86803369058214690\,x^5 \,-\,61120624939489989\,x^4 \,-\,30036033003931493\,x^3 \,-\,5927897678515792\,x^2 \,-\,271244487735336\,x \,-\,\\ 1643001862841472$.\\ \normalsize

\noindent Derived PLRR: $H_{n+1}\, =\, H_n \,+\, 17581679276200473\,H_{n-654} \,+\, 43065699679149511\,H_{n-655} \,\\+ \,70765959937154578\,H_{n-656}\, +\, 91624450164084254\,H_{n-657} \,+\, 98016133194347743\,H_{n-658}\,\\
+\,86803369058214690\,H_{n-659}\,+\,61120624939489989\,H_{n-660} \,+\,30036033003931493\,H_{n-661}\\ \,+\,5927897678515792\,H_{n-662}\,+\,271244487735336\,H_{n-663} \,+\,1643001862841472\,H_{n-664}$.\\

\noindent \textbf{17.} Current ZLRR: $G_{n+1}=G_{n-1}\,+\,G_{n-6}$.\\

\noindent Current characteristic polynomial: $x^{7}\,-\,x^5\,-\,1$.\\

\noindent Derived characteristic polynomial: $x^{37}\,-\,x^{36}\,-\,18\,x^6 \,-\,2\,x^5 \,-\,9\,x^4 \,-\,2\,x^3 \,-\,7\,x^2 \,-\,9\,x \,-\,4$.\\

\noindent Derived PLRR: $H_{n+1} = H_n\,+\, 18\,H_{n-30}\,+\,2\,H_{n-31}\,+\,9\,H_{n-32}\,+\,2\,H_{n-33}\,+\,7\,H_{n-34}\,+\,9\,H_{n-35}\,+\,4\,H_{n-36}$.\\

\noindent \textbf{18.} Current ZLRR: $G_{n+1}=2G_{n-2}\,+\,3G_{n-3}\,+\,5G_{n-5}$.\\

\noindent Current characteristic polynomial: $x^6\,-\,2\,x^3\,-\,3\,x^2\,-\,5$.\\

\noindent Derived characteristic polynomial: $x^{19}\,-\,x^{18}\,-\,75\,x^5 \,-\,207\,x^4 \,-\,708\,x^3 \,-\,384\,x^2 \,-\,370\,x \,-\,740$.\\

\noindent Derived PLRR: $H_{n+1} = H_n\,+\, 75\,H_{n-13}\,+\,207\,H_{n-14}\,+\,708\,H_{n-15} \,+\, 384\,H_{n-16}\,+\,370\,H_{n-17}\,+\,740\,H_{n-18}$.\\

\noindent \textbf{19.} Current ZLRR: $G_{n+1} = G_{n-1}\,+2\,G_{n-2}$.\\

\noindent Current characteristic polynomial: $x^3\,-\,x\,-\,2$.\\

\noindent Derived characteristic polynomial: $x^8\,-\,x^7\,-\,x^2\,-\,x\,-\,6$. Derived PLRR: $H_{n+1}=H_n\,+\, H_{n-5}\,+\,H_{n-6}\,+\,6H_{n-7}$.
\vspace{-1mm}

\newpage

%%%%%%%%%%%%%%%%%%%%%%%%%%%%%%%%%%%%%%%%%%%%%%%%%%%%%%%%%%%%%%%%%%%%%%%%%%%%%%%%%%%%%%%%%%%%%%%%%%%%%%%%%%%%%%%%%%%%%%%%%%%%%%%%%%%%%%%%%%%%%%%%%%%%%%
%%%%%%%%%%%%%%%%%%%%%%%%%%%%%%%%%%%%%%%%%%%%%%%%%%%%%%%%%%%%%%%%%%%%%%%%%%%%%%%%%%%%%%%%%%%%%%%%%%%%%%%%%%%%%%%%%%%%%%%%%%%%%%%%%%%%%%%%%%%%%%%%%%%%%%
%%%%%%%%%%%%%%%%%%%%%%%%%%%%%%%%%%%%%%%%%%%%%%%%%%%%%%%%%%%%%%%%%%%%%%%%%%%%%%%%%%%%%%%%%%%%%%%%%%%%%%%%%%%%%%%%%%%%%%%%%%%%%%%%%%%%%%%%%%%%%%%%%%%%%%
%%%%%%%%%%%%%%%%%%%%%%%%%%%%%%%%%%%%%%%%%%%%%%%%%%%%%%%%%%%%%%%%%%%%%%%%%%%%%%%%%%%%%%%%%%%%%%%%%%%%%%%%%%%%%%%%%%%%%%%%%%%%%%
%%%%%%%%%%%%%%%%%%%%%%%%%%%%%%%%%%%%%%%%%%%%%%%%%%%%%%%%%%%%%%%%%%%%%%%%%%%%%%%%%%%%%%%%%%%%%%%%%%%%%%%%%%%%%%%%%%%%%%%%%%%%%%

\ \\

MSC2010: 11B39, 65Q30


\begin{thebibliography}{CFHMNPX} % '2nd argument contains the widest acronym'
\bibitem[BBGILMT]{BBGILMT}
O. Beckwith, A. Bower, L. Gaudet, R. Insoft, S. Li, S. J. Miller and P. Tosteson, \emph{The Average Gap Distribution for Generalized Zeckendorf Decompositions}, Fibonacci Quarterly \textbf{51} (2013), 13--27.

\bibitem[BM]{BM}
I. Ben-Ari, S. Miller, \emph{A Probabilistic Approach to Generalized Zeckendorf Decompositions}, SIAM Journal on Discrete Mathematics, \textbf{30} (2016), no. 2, 1302-1332.

\bibitem[BCCSW]{BCCSW}
E. Burger, D.C. Clyde, C.H. Colbert, G.H. Shin, Z. Wang \emph{A generalization of a theorem of Lekkerkerker to Ostrowski's decomposition of natural numbers}
Acta Arithmetica, \textbf{153} (2012), pp. 217-249.

\bibitem[CFHMN]{CFHMN}
M. Catral, P. Ford, P. E. Harris, S. J. Miller, and D. Nelson, \emph{Generalizing Zeckendorf's Theorem: The
Kentucky Sequence}, Fibonacci Quarterly \textbf{52} (2014), no. 5, 68-90.

\bibitem[CFHMNPX]{CFHMNPX}
M. Catral, P. Ford, P. E. Harris, S. J. Miller, D. Nelson, Z. Pan and H. Xu, \emph{New Behavior in Legal Decompositions Arising from Non-positive Linear Recurrences}, Fibonacci Quarterly \textbf{55} (2017), no. 3, 252-275.

\bibitem[DFFHMPP]{DFFHMPP}
R. Dorward, P. Ford, E. Fourakis, P. Harris, S. Miller, E. Palsson, H. Paugh, \emph{New Behavior in Legal Decompositions Arising From Non-Positive Linear Recurrences}, Fibonacci Quarterly, \textbf{55} (2017), no. 3, 252-275.

\bibitem[DG]{DG}
M. Drmota and J. Gajdosik, \emph{The distribution of the sum-of-digits function}, J. Th\'{e}or. Nombr\'{e}s
Bordeaux \textbf{10} (1998), no. 1, 17-32.

\bibitem[Go]{Go}
S. Goldberg, \emph{Introduction to Difference Equations}, John Wiley \& Sons, 1961.

\bibitem[GT]{GT}
P. J. Grabner and R. F. Tichy, \emph{Contributions to digit expansions with respect to linear recurrences}, J. Number Theory \textbf{36} (1990), no. 2, 160-169.

\bibitem[Ho]{Ho}
V. E. Hoggatt, \emph{Generalized Zeckendorf theorem}, Fibonacci Quarterly \textbf{10} (1972), no. 1 (special issue
on representations), pages 89-93.

\bibitem[Ke]{Ke}
T. J. Keller, \emph{Generalizations of Zeckendorf's theorem}, Fibonacci Quarterly \textbf{10} (1972), no. 1 (special
issue on representations), pages 95-102.

% Determinant of "Punctured" Vandermonde Matrices
	\bibitem[KKL]{KKL} N. Kolokotronis, K. Limniotis, and N. Kalouptsidis. \textit{Lower Bounds on Sequence Complexity Via Generalised Vandermonde Determinants.} Sequences and Their Applications - SETA 2006. Vol. 4086. Berlin, Heidelberg: Springer Berlin Heidelberg, 2006. 271-84.

\bibitem[KKMW]{KKMW}
M. Kolo\u{g}lu, G. Kopp, S. J. Miller and Y. Wang, \emph{On the number of summands in Zeckendorf Decompositions}, Fibonacci Quarterly \textbf{49} (2011), no. 2, 116-130.

\bibitem[Len]{Len}
T. Lengyel, \emph{A Counting Based Proof of the Generalized Zeckendorf's Theorem}, Fibonacci Quarterly \textbf{44} (2006), no. 4, 324-325.

\bibitem[MMMS1]{MMMS1}
T. C. Martinez, S.J. Miller, C. Mizgerd and C. Sun, \emph{Generalizing Zeckendorf's Theorem to Homogeneous Linear Recurrences, I}, preprint, \url{https://arxiv.org/pdf/2001.08455.pdf}

\bibitem[MNPX]{MNPX}
S. J. Miller, D. Nelson, Z. Pan and H. Xu, \emph{On the Asymptotic Behavior of Variance of PLRS Decompositions}, preprint, \url{https://arxiv.org/pdf/1607.04692.pdf}

\bibitem[MT-B]{MT-B}
S. J. Miller and R. Takloo-Bighash, \emph{An Invitation to Modern Number Theory}, Princeton University Press, Princeton, NJ, 2006.

\bibitem[MW]{MW}
S. J. Miller and Y. Wang, \emph{From Fibonacci numbers to Central Limit Type Theorems}, Journal of Combinatorial Theory, Series A \textbf{119} (2012), no. 7, 1398-1413.

\bibitem[Ste]{Ste}
W. Steiner, \emph{Parry expansions of polynomial sequences}, Integers \textbf{2} (2002), Paper A14.

\bibitem[Ze]{Ze} E. Zeckendorf, \emph{Repr\'{e}sentation des nombres naturels par une somme des nombres de Fibonacci ou
de nombres de Lucas}, Bulletin de la Soci\'{e}t\'{e} Royale des Sciences de Li\`{e}ge \textbf{41} (1972), pages 179-182.


\end{thebibliography}
\end{document}